\documentclass[12pt,leqno,a4paper]{article}
\usepackage{graphicx}
\usepackage{caption}
\usepackage{subcaption}
\usepackage{color}

\usepackage[pdfborder={0 0 0}]{hyperref}

\usepackage{amssymb}
\usepackage{amsmath}
\usepackage{amsthm}
\usepackage{mathrsfs}

\usepackage[shortlabels]{enumitem}

\newtheorem{theorem}{Theorem}
\numberwithin{theorem}{section}

\newtheorem{lemma}[theorem]{Lemma}
\newtheorem{corollary}[theorem]{Corollary}

\theoremstyle{remark}

\theoremstyle{definition}
\newtheorem{remark}[theorem]{Remark}
\newtheorem{definition}[theorem]{Definition}

\numberwithin{equation}{section}

\newcommand\set[1]{\left\{\,#1\,\right\}}		
\newcommand\abs[1]{\left|#1\right|}				
\newcommand\ska[1]{\left\langle#1\right\rangle} 
\newcommand\norm[1]{\left\Vert#1\right\Vert}	

\DeclareMathOperator{\dist}{dist}				
\DeclareMathOperator{\diam}{diam}				
\DeclareMathOperator{\id}{id}					
\DeclareMathOperator{\supp}{supp}				
\DeclareMathOperator{\tr}{tr}					
\DeclareMathOperator{\divv}{div}				
\newcommand\hoeld[1]{{\cC^{\,#1\,}(\overline{Q})}}

\def\N{\mathbb{N}}
\def\Z{\mathbb{Z}}

\def\R{\mathbb{R}}

\newcommand{\cC}{{\mathcal C}}

\newcommand{\cI}{{\mathcal I}}

\newcommand{\cO}{{\mathcal O}}

\newcommand{\cQ}{{\mathcal Q}}

\usepackage[numbers,sort]{natbib}

\begin{document}

\title{A general way to confined stationary Vlasov-Poisson plasma configurations}
\author{Yulia O. Belyaeva\thanks{ Supported by the Russian Foundation for Basic Research (grant 20--01--00288).}\and  Bj\"orn Gebhard\thanks{Supported by the German-Russian Interdisciplinary Science Center (G-RISC) funded by the German Federal Foreign Office via the German Academic Exchange Service (DAAD) (Project numbers: M-2018b-2, A-2019b-5\_d).}\and  Alexander L. Skubachevskii$^*$}
\date{}
\maketitle

\begin{abstract}
We address the existence of stationary solutions of the Vlasov-Poisson system on a domain $\Omega\subset\mathbb{R}^3$ describing a high-temperature plasma which due to the influence of an external magnetic field is spatially confined to a subregion of $\Omega$. In a first part we provide such an existence result for a generalized system of Vlasov-Poisson type and investigate the relation between the strength of the external magnetic field, the sharpness of the confinement and the amount of plasma that is confined measured in terms of the total charges. The key tools here are the method of sub-/supersolutions and the use of first integrals in combination with cutoff functions. In a second part we apply these general results to the usual Vlasov-Poisson equation in three different settings: the infinite and finite cylinder, as well as domains with toroidal symmetry. This way we prove the existence of stationary solutions corresponding to a two-component plasma confined in a Mirror trap, as well as a Tokamak. 
\end{abstract}
%
%

\section{Introduction}\label{sec:introduction}
The Vlasov-Poisson system describing a two-component plasma 
under the influence of an external magnetic field reads
\begin{gather}\label{eq:traditional_V}
\partial_tf^\beta+\ska{v,\nabla_xf^\beta}_{\R^3}+\frac{q_\beta}{m_\beta}\ska{-\nabla_x\varphi+\frac{v\times B}{c},\nabla_vf^\beta}_{\R^3}=0,
\end{gather}
\begin{gather}\label{eq:traditional_P}
-\Delta\varphi(x,t)=4\pi\sum\limits_{\beta=\pm} q_\beta\int_{\R^3}f^\beta\:dv. 
\end{gather} 
Here $f^\beta=f^\beta(x,v,t)\geq 0$, $\beta\in\set{+,-}$,
is the distribution function
of positively charged ions
(for $\beta=+$) and electrons (for $\beta=-$) resp., at a point $x$ with velocity $v\in\R^3$
at the time $t\in[0,T)$. Depending on the application the system can be considered for $x \in \R^3$, i.e. on the whole space, or, as will be the case in the present paper, on a domain $\Omega\subset\R^3$ whose boundary we suppose to be sufficiently smooth. In the case of a domain the Poisson equation \eqref{eq:traditional_P} for the potential $\varphi:\overline{\Omega}\times[0,T)\rightarrow\R$ of the self-consistent electric field will be complemented by the Dirichlet boundary condition
\begin{equation}\label{eq:trad_BC}
\varphi(x,t)=0,\quad (x,t)\in\partial\Omega\times[0,T),
\end{equation}
modelling a perfectly conducting reactor wall,
and initial conditions
\begin{equation}\label{e1.4}
f^\beta(x,v,t)_{|t=0}=f_0^\beta(x,v).
\end{equation}
The parameters $m_{\beta}>0$, $\beta=\pm$, are the masses of the charged particles, $q_-<0$ is the charge of the electrons, $q_+>0$ the charge of the ions, $c>0$ is the speed of light and $B:\overline{\Omega}\times[0,T)\rightarrow\R^3$ is the external magnetic field influencing the particle trajectories through the Lorentz force, and $f_0^\beta(x,v)$ is the initial distribution function.

The Vlasov equations in general have a broad range of applications in various fields of physics. They have been derived in 1938 \cite{Vlasov} by Vlasov, who justified that in the kinetic description of a high-temperature plasma collisions between different particles can be neglected. Neglecting the corresponding collision term in the Boltzmann equation he obtained the Vlasov-Maxwell equations, from which one derives the Vlasov-Poisson system by additionally neglecting the self-consistent magnetic field generated by the motion of the charged particles. 

One particular application in plasma physics is the construction of a reactor for controlled thermonuclear fusion, see for example \cite{HazeltineMeiss,Miyamoto,Stacey}. Among such devices are reactor chambers based on a toroidal form, like Tokamaks and Stellarators, and a cylindrical  form, like Mirror traps or z-pinch devices. Due to the high temperature of the plasma a key feature in the realization of such a reactor is the use of an external magnetic field, which has to be choosen in a way, such that the plasma is strictly confined in the interior of the chamber.

In the present paper we prove within the Vlasov-Poisson description \eqref{eq:traditional_V}, \eqref{eq:traditional_P}, \eqref{eq:trad_BC} of a two-component plasma the existence of stationary configurations that are strictly confined in a chamber of Tokamak type, as well as stationary solutions corresponding to a confinement in a Mirror trap device. In view of the macroscopic charge neutrality of a high-temperature plasma it is important to consider a two-component system modelling positive and negative charged particles, but our results also apply to the one-component case and more generally even to a plasma consisting of $N\in\N$ different types of particles. In any of these cases, to the best of our knowledge, there has been so far no existence result concerning stationary confined solutions in toroidal domains, nor in Mirror traps.
Before providing further details, let us quickly give a short, due to the extensive amount of literature, not complete, survey on the topic.

In the case $\Omega=\R^3$ the Poisson equation 
\eqref{eq:traditional_P} can be solved using the Newtonian potential. 
Substituting the convolution into the Vlasov equation \eqref{eq:traditional_V} we get 
an integro-differential equation with singular kernel.  
At first global solvability has been obtained in \cite{BraunHepp,Dobrushin,Maslov}
for the corresponding equation with a ``smoothed'' kernel.
The existence and further properties of global generalized solutions of the Cauchy problem for the actual Vlasov-Poisson
equations have been studied by A.A.Arsen'ev \cite{Arsen'ev}, R.J.DiPerna, P.L.Lions \cite{DiPernaLions}, and E.Horst, R.Hunze \cite{HorstHunze},  while 
global classical solutions have been investigated in the papers of C.Bardos, P.Degond \cite{Bardos}, J.Batt \cite{Batt1977}, E.Horst \cite{Horst}, K.Pfaffelmoser  \cite{Pfaffelmoser}, and J.Schaeffer \cite{Schaffer}. The initial-boundary value problems for classical solutions have also been studied in a half-space under reflection conditions at the boundary, see \cite{Guo,HwangVelazquez}.

Concerning the confinement problem for a two--component plasma the articles \cite{BelSkub20,SkubDAN,Skub2013,SkubUMN,SkubTsu} provide quantitative estimates for a bounded magnetic field guaranteeing the classical solution to the initial-boundary value problem of \eqref{eq:traditional_V}, \eqref{eq:traditional_P}, \eqref{eq:trad_BC}, \eqref{e1.4} to exist and to be located at some distance from the boundary of a half-space and an infinite cylinder.

Unlike \cite{BelSkub20,SkubDAN,Skub2013,SkubUMN,SkubTsu}, the papers 
\cite{CaprinoCavallaroMarchioro_cylinder,CaprinoCavallaroMarchioro_torus} deal with the confinement problem for a one--component plasma with  an external magnetic field, which becomes infinite on a boundary. Equation \eqref{eq:traditional_P} is studied in the whole space $\R^3$, and the influence of boundary conditions to a solution of equation \eqref{eq:traditional_P} is not considered. This allows to apply the approach of K.Pfaffelmoser  \cite{Pfaffelmoser} and J.Schaeffer \cite{Schaffer} in order to obtain a global existence result for the Cauchy problem.

Stationary solutions of the Vlasov-Poisson equations have been studied in various settings
\cite{BattFaltenbacherHorst,Batt2019,MMNP,GreengardRaviart,
Knopf,Pokhozhaev,SkubUMN,Rein,Vedenyapin1,Vedenyapin2,Weber}. 
Let us focus on the ones addressing the confinement problem. The existence of stationary solutions to \eqref{eq:traditional_V}, \eqref{eq:traditional_P}, \eqref{eq:trad_BC} with vanishing potential and density distribution functions supported away from the considered
boundary, as well as compactly supported distribution functions have first been shown to exist for $\Omega$ being an infinite cylinder and a half-space in  \cite{SkubUMN,MMNP}. 
On $\Omega=\R^3$ stationary solutions confined to an infinite cylinder and with the Newtonian electric potential have been constructed in \cite{Knopf}. In \cite{Weber} stationary confined solutions in an infinite cylinder have also been constructed for the relativistic Vlasov-Maxwell system. The proofs in \cite{Knopf,Weber} rely, after a suitable ansatz and reduction, on a fixed point argument.

In the above articles, as well as in the present, the external magnetic field is fixed, but we also like to mention the different approach in \cite{Knopf_control,KnopfWeber_control,Weber_control}, where the external magnetic field is viewed as part of an optimal control problem.

In this paper we consider stationary solutions of the Vlasov--Poisson system for a multi--component plasma with the Dirichlet boundary condition for the electric potential. We will do this first for a general system of Vlasov-Poisson type and then apply the outcome to three specific settings: the infinite and finite cylinder, as well as domains with toroidal symmetry. Using first integrals and cutoff functions, we reduce the problem of finding confined stationary solutions to a quasilinear elliptic differential equation with Dirichlet boundary condition. Finally, the existence of stationary solutions with compactly supported distribution functions is proven based on the method of sub-- and supersolutions for the first boundary value problem for quasilinear elliptic equations, see \cite{ako_dirichlet_1961}. We note that in our case the magnetic field does not need to be singular, and the electric potential, generally speaking, is not trivial.

In the remaining part of the introduction we will illustrate our results based on one specific case, which will be the ``Tokamak-case'' (Figure 1).

\begin{figure}
\centering\includegraphics[width=12cm]{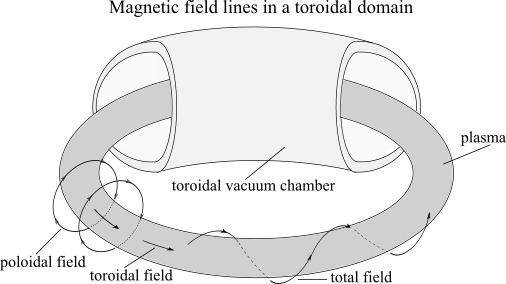}
\caption{}
\end{figure}

The group $S^1=\R/2\pi\Z$ acts isometrically on $\R^3$ via rotations around the $x_3$-axis. For $\theta\in S^1$ and $x\in\R^3$, this action is denoted by
\[
\theta *x=\begin{pmatrix}
\cos\theta & -\sin\theta & 0\\
\sin\theta & \cos \theta & 0\\
0 & 0 & 1
\end{pmatrix}\begin{pmatrix}
x_1\\x_2\\x_3
\end{pmatrix}.
\]
Let now $\Omega\subset\R^3$ be a smooth bounded domain which is invariant under the $S^1$-action and which does not contain a point of the form $(0,0,x_3)$ in its closure.

For $x_0\in\R^3$, $\delta>0$, we define the toroidal neighborhoods
\begin{equation}\label{eq:definition_of_toroidal_nbhd}
\cO_\delta(x_0):=\set{x\in\R^3:\dist\left(x,S^1*x_0\right)<\delta}=S^1*B_\delta(x_0).
\end{equation}

A function $f:\overline{\Omega}\times\R^3\rightarrow \R$, $\varphi:\overline{\Omega}\rightarrow\R$ resp., is said to be $S^1$-invariant provided $f(\theta*x,\theta*v)=f(x,v)$, $\varphi(\theta*x)=\varphi(x)$ resp., for all $\theta\in S^1$, $x\in\overline{\Omega}$, $v\in\R^3$.
Similar a vector field $B:\overline{\Omega}\rightarrow\R^3$ is $S^1$-equivariant if $B(\theta*x)=\theta*B(x)$ for all $\theta\in S^1$, $x\in\overline{\Omega}$.

For $x\in\overline{\Omega}$, let $P_x:\R^3\rightarrow\R^3$ be the orthogonal projection onto the plane spanned by the vectors $(x_1,x_2,0)$ and $(0,0,1)$. We use this projection to decompose a magnetic field $B:\overline{\Omega}\rightarrow\R^3$ into its poloidal part $P_xB(x)$ and its toroidal part $(\id_{\R^3}-P_x)B(x)$.

Let $x_0\in\Omega$ and set $r_0:=\sqrt{x_{0,1}^2+x_{0,2}^2}>0$, $z_0:=x_{0,3}$. For the confinement of the spatial supports of $f^+$, $f^-$ we rely on the  magnetic field $B^{x_0}:\overline{\Omega}\rightarrow \R^3$,
\[
B^{x_0}(x):=\frac{1}{x_1^2+x_2^2}\begin{pmatrix}
x_1(x_3-z_0)\\
x_2(x_3-z_0)\\
-\sqrt{x_1^2+x_2^2}\left(\sqrt{x_1^2+x_2^2}-r_0\right)
\end{pmatrix}.
\]
Note that $P_xB^{x_0}(x)=B^{x_0}(x)$, i.e. $B^{x_0}$ is a poloidal field. One can also directly compute that $B^{x_0}$ is divergence-free.

For clarification, by a stationary solution of the Vlasov-Poisson system on $\Omega$ we understand a triple $(f^+,f^-,\varphi)$ of time independent functions $f^\pm\in\cC^1(\overline{\Omega}\times\R^3)$, $f^\pm\geq 0$, $\varphi\in\cC^2(\overline{\Omega})$ with $\int_{\R^3}f^\pm(\cdot,v)\:dv\in\cC^0(\overline{\Omega})$ and such that these functions satisfy the equations \eqref{eq:traditional_V} (with $\partial_tf^\beta=0$), $\beta =\pm$, as well as the boundary value problem  \eqref{eq:traditional_P}, \eqref{eq:trad_BC}.

The total charge of the $\beta$th component of a stationary solution $(f^+,f^-,\varphi)$ is defined by
$\cQ^\beta:=q_\beta\norm{f^\beta}_{L^1(\Omega\times\R^3)}$.

Our main result for the two-component Vlasov-Poisson system \eqref{eq:traditional_V}, \eqref{eq:traditional_P}, \eqref{eq:trad_BC} considered on the toroidal domain $\Omega$ reads as follows.
\begin{theorem}\label{thm:toroidal_solutions}
Let $x_0\in\Omega$ and $B_b:\overline{\Omega}\rightarrow\R^3$ be a $S^1$-equivariant magnetic field with poloidal part $P_xB_b(x)=bB^{x_0}(x)$, where $b>0$ is a parameter.
\begin{enumerate}[(i)]
\item Let $b>0$ be fixed. Then to any collection of numbers $0<\delta_\pm<\dist(x_0,\partial\Omega)$, $\varepsilon_\pm>0$, $c>0$ there exists a $S^1$-invariant stationary solution $\left(f^+,f^-,\varphi\right)$ of the Vlasov-Poisson system \eqref{eq:traditional_V}, \eqref{eq:traditional_P}, \eqref{eq:trad_BC} considered with $B_b$, such that  $\supp f^\pm\subset \cO_{\delta_\pm}(x_0)\times B_{\varepsilon_\pm}(0)$ and $\cQ^+=c\abs{\cQ^-}>0$.
\item Let $(f^+,f^-,\varphi)$ be a solution from (i) associated with the parameter values $b,\delta_\pm,\varepsilon_\pm,c$. Then for any $\lambda\in(0,\infty)$ the Vlasov-Poisson system \eqref{eq:traditional_V}, \eqref{eq:traditional_P}, \eqref{eq:trad_BC} considered with magnetic field $B_{\lambda b}$ has a stationary solution $(f^+_\lambda,f^-_\lambda,\varphi_\lambda)$ with $\supp f^\pm_\lambda\subset \cO_{\delta_\pm}(x_0)\times B_{\lambda\varepsilon_{\pm}}(0)$ and total charges $\cQ_\lambda^\pm=\lambda^2\cQ^\pm$.
\end{enumerate}
\end{theorem}

Concerning this Theorem we like to point out that the toroidal part of the magnetic field does not play a role for the existence of confined stationary solutions. Moreover, as the analysis in Section \ref{sec:toroidal_sym} will show, a purely toroidal field, which would be the analogue to the constant magnetic field used in \cite{SkubUMN,Knopf} for the cylindrical case, cf. Section \ref{sec:infinite_cylinder}, can not be used to guarantee the existence of stationary solutions with spatial supports strictly contained in $\Omega$. 
These two observations for the Vlasov-Poisson description of a tokamak plasma are in agreement with the general physical understanding: ``Thus, in terms of simple force balance, the poloidal field does most of the work in tokamak confinement. The toroidal field enhances stability, as well as improving thermal insulation.'' -- \cite[Section 1.5]{HazeltineMeiss}.

However, besides existence alone, the question in which sense the toroidal part influences the stability of the found solutions is a different one, which as the broader question concerning stability in general is not addressed in the present paper.

Part (i) of Theorem \ref{thm:toroidal_solutions} shows that the strength of the external magnetic field, which corresponds to the size of the parameter $b>0$, is not important if one is only interested in the existence of some stationary solutions with supports in a prescribed region. Only if one wants to confine a given amount of plasma, measured in terms of the total charges $\cQ^\pm$, a sufficiently strong magnetic field becomes crucial, cf. part (ii).
In fact combining the two parts of Theorem \ref{thm:toroidal_solutions} one obtains
\begin{corollary}
To any choice of $\delta_\pm\in(0,\dist(x_0,\partial\Omega))$, $c_+>0$, $c_-<0$ there exists $b>0$, such that \eqref{eq:traditional_V}, \eqref{eq:traditional_P}, \eqref{eq:trad_BC} considered with the magnetic field $B_b$ has a stationary solution $(f^+,f^-,\varphi)$ with $\supp f^\pm\subset\subset \cO_{\delta_\pm}(x_0)\times\R^3$ and $\cQ^\pm=c_\pm$.
\end{corollary}

Note also that in particular if $\cQ^+\neq \abs{\cQ^-}$, the electric potential $\varphi$ is non-trivial. Moreover, the results extend to the extreme cases $\cQ^+=0$, $\cQ^-<0$ and $\cQ^+>0$, $\cQ^-=0$, i.e. we also find confined stationary solutions for the two different one-component systems modelling a plasma consisting only of ions or electrons.

A similar result holds true for solutions in an infinite and finite cylinder, which in order to avoid repetition we do not formulate here, but refer to the corresponding Sections \ref{sec:infinite_cylinder}, \ref{sec:finite_cylinder} instead.

The paper is organized as follows. In Section \ref{sec:general_VP} we consider a generalized Vlasov-Poisson system and prove the existence of stationary solutions with controlled velocity support under quite mild conditions. Section  \ref{sec:further_properties} still treats a general system but under additional assumptions allowing now also a confinement of the spatial supports, see \ref{sec:confinement_in_space}, as well as an investigation of the relation between the total charges and the strength of the magnetic field, see \ref{sec:scaling_of_total_charges} and \ref{sec:ratio_of_charges}. After this we turn to the application of these general results to specific settings. We begin in Section \ref{sec:toroidal_sym} with toroidally symmetric solutions including the proof of Theorem \ref{thm:toroidal_solutions}, continue with the infinite cylinder in Section \ref{sec:infinite_cylinder},  and end in Section \ref{sec:finite_cylinder} with Mirror trap type solutions in a finite cylinder.

\section{A generalized Vlasov-Poisson system}\label{sec:general_VP}
We consider on a smooth bounded domain $Q\subset\R^n$ a Vlasov-Poisson type system for $N\in\N$ types of particles with charge $q_\beta\in\R\setminus\{0\}$ and mass $m_\beta>0$, $\beta=1,\ldots,N$. On the time interval $[0,T)$ the distribution of the particles of type $\beta\in\{1,\ldots,N\}$ is described by a density distribution function $f^\beta:\overline{Q}\times\R^m\times[0,T)\rightarrow[0,\infty)$, $(x,v,t)\mapsto f^\beta(x,v,t)$  obeying the Vlasov type equation
\begin{equation}\label{eq:generalized_vlasov_equation}
\partial_tf^\beta+\ska{M^\beta v,\nabla_xf^\beta}_{\R^n}-\frac{q_\beta}{m_\beta}\ska{\nabla_x\varphi,M^\beta \nabla_v f^\beta}_{\R^n}+\ska{F^\beta,\nabla_vf^\beta}_{\R^m}=0,
\end{equation}
where $M^\beta:\overline{Q}\times\R^m\times[0,T)\rightarrow\R^{n\times m}$, $(x,v,t)\mapsto M^\beta(x,v,t)$, $F^\beta:\overline{Q}\times\R^m\times[0,T)\rightarrow\R^m$, $(x,v,t)\mapsto F^\beta(x,v,t)$ are given continuous functions, $\beta=1,\ldots,N$ and $\varphi:\overline{Q}\times[0,T)\rightarrow\R$ is the self-consistent generalized electric potential satisfying
\begin{equation}\label{eq:generalized_poisson_equation}
-L\varphi=4\pi\sum_{\beta=1}^Nq_\beta\rho^\beta \text{ in }Q\times [0,T)\quad\text{and}\quad \varphi=g\text{ on }\partial Q\times[0,T).
\end{equation}
Here $L=a_{ij}(x)\partial_{x_i}\partial_{x_j}+b_i(x)\partial_{x_i}$ with $a_{ij},b_i:\overline{Q}\rightarrow\R$, $1\leq i,j\leq n$, is a second order differential operator, $g:\overline{Q}\times[0,T)\rightarrow\R$ is given boundary data and the functions $\rho^\beta:\overline{Q}\times[0,T)\rightarrow[0,\infty)$ are the spatial densities induced by $f^\beta$, i.e.,  for $\beta=1,\ldots,N$ we have
\begin{equation}\label{eq:definition_of_spatial_densities}
\rho^\beta(x,t):=\int_{\R^m}f^\beta(x,v,t)\:dv.
\end{equation}
\begin{remark}
a) The usual two-component Vlasov-Poisson system in the domain $Q$ and with an external magnetic field $B:\overline{Q}\times[0,T)\rightarrow\R^3$, as stated in the introduction is obtained by setting $N=2$, $q_1=q_-<0$, $q_2=q_+>0$, as well as $n=m=3$, $M^\beta\equiv \id_{\R^3}$, $L=\Delta=\partial_{x_1}^2+\partial_{x_2}^2+\partial_{x_3}^2$, $g\equiv 0$. The external magnetic field $B:\overline{Q}\times[0,T)\rightarrow\R^3$ acts on the system via the Lorentz force $F^\beta(x,v,t)=\frac{q_\beta}{cm_\beta}v\times B(x,t)$, where $c>0$ denotes the speed of light.

b) Some parts of the here considered generalizations have purely mathematical reasons, while some other parts are needed in our applications, e.g. the asymmetry between spatial and velocity dimensions or the deviation of $L$ from being simply the Laplace operator naturally appear when investigating solutions in a domain with toroidal symmetry, cf. Section \ref{sec:toroidal_sym}.

c) Concerning further generalizations, we can even allow $M^\beta$ and $F^\beta$ to depend on $\varphi$ or on $(f^\beta)_{\beta=1}^N$ as the proof of our abstract Theorem \ref{thm:abstract_theorem} shows. This way control terms can be included in the model. Also $L$ can be an even more complicated second order differential operator as long as it satisfies the conditions of Ak\^{o}'s paper \cite{ako_dirichlet_1961}, on which our proof relies.
\end{remark}

\subsection{Existence of stationary solutions}
We are interested in classic stationary solutions to the system \eqref{eq:generalized_vlasov_equation}--\eqref{eq:definition_of_spatial_densities}, by which we mean the following.
\begin{definition}\label{def:stationary_solution} A stationary solution to the generalized Vlasov-Poisson system is a tuple $(f^1,\ldots,f^N,\varphi)$ of time independent functions $f^\beta\in\cC^1(\overline{Q}\times\R^m)$, $f^\beta\geq 0$, $\varphi\in\cC^2(\overline{Q})$, such that each $\rho^\beta=\int_{\R^m}f^\beta(\cdot,v)\:dv$ is continuous on all of $\overline{Q}$ and such that these functions satisfy the equations \eqref{eq:generalized_vlasov_equation} (with $\partial_tf^\beta=0$), $\beta =1,\ldots,N$, and the boundary value problem  \eqref{eq:generalized_poisson_equation}.
\end{definition}

Stationary solutions of course can only exist provided the boundary data $g$ is independent of time $t$, whereas a time dependence in the functions $M^\beta$ and $F^\beta$ is still allowed.

The strategy to find stationary solutions  also here is based on the well-known method, see \cite{BattFaltenbacherHorst,Batt2019,Pokhozhaev,SkubUMN,Vedenyapin1,Vedenyapin2,
Weber}, of exploiting first integrals of the characteristic system
\begin{align}\label{eq:generalized_characteristic_system}
\dot{x}=M^\beta(x,v,t)v,\quad \dot{v}=-\frac{q_\beta}{m_\beta}(M^\beta(x,v,t))^T\nabla_x\varphi(x,t)+F^\beta(x,v,t)
\end{align}
associated with \eqref{eq:generalized_vlasov_equation}. Here $(M^\beta)^T$ denotes the transposed matrix. Indeed, if for fixed $\varphi\in\cC^2(\overline{\Omega}\times[0,T))$ a function $I\in\cC^1(\overline{Q}\times\R^m)$ is constant along solutions of \eqref{eq:generalized_characteristic_system}, then $I$ solves the Vlasov equation \eqref{eq:generalized_vlasov_equation} considered with that $\varphi$.

\begin{definition}\label{def:set_of_all_integrals} For $\beta=1,\ldots,N$, we define $\cI^\beta$ to be the set of all $\cC^1$ functions $I^\beta:\overline{Q}\times\R^m\times\R\rightarrow\R$, $(x,v,u)\mapsto I^\beta(x,v,u)$, such that for all stationary potentials $\varphi\in\cC^2(\overline{Q})$ the function $\overline{Q}\times\R^m\ni(x,v)\mapsto I^\beta(x,v,\varphi(x))\in\R$ is a  first integral to \eqref{eq:generalized_characteristic_system} considered with that $\varphi$.
\end{definition}

A natural candidate for a first integral of \eqref{eq:generalized_characteristic_system} is the energy
$
\frac{1}{2}m_\beta\abs{v}^2+q_\beta\varphi(x),
$
which corresponds to the $x$-independent function $E^\beta:\R^m\times\R\rightarrow\R$,
\[
E^\beta(v,u)=\frac{1}{2}m_\beta\abs{v}^2+q_\beta u.
\]
Indeed simple differentiation shows
\begin{lemma}\label{lem:energy_is_first_integral}
Under the condition
\begin{equation}\label{eq:condition_on_F}
\ska{F^\beta(x,v,t),v}_{\R^m}=0 \text{ for all }(x,v,t)\in\overline{Q}\times\R^m\times[0,T),
\end{equation}
there holds $E^\beta\in\cI^\beta$, $\beta=1,\ldots,N$.
\end{lemma}

Note that condition \eqref{eq:condition_on_F} for example holds true in the case that $F^\beta$ is given by a Lorentz force $F^\beta(x,v,t)=\frac{q_\beta}{cm_\beta}v\times B(x,t)$.

In order to solve the generalized Poisson equation \eqref{eq:generalized_poisson_equation} we will rely on the method of sub- and supersolutions. In particular we will use the main theorem of the article \cite{ako_dirichlet_1961} by Ak\^o. The result presented there fits very well the search for potentials $\varphi$ of class $\cC^2$. For other aspects of the sub-/supersolution method in the investigation of nonlinear Poisson equations, for example in the context of weak solutions, we refer the reader to \cite{MontenegroPonce} and the references therein. We now state conditions on the differential operator $L$ and the boundary data $g$ that allow us to apply Ak\^o's result. In fact compared to the actual formulation in \cite{ako_dirichlet_1961} we consider a slightly simpliefied setting. As stated before we assume that $L$ has the form
\begin{equation}\label{eq:form_of_L}
L=\sum_{i,j=1}^na_{ij}(x)\partial_{x_i}\partial_{x_j}+\sum_{i=1}^nb_i(x)\partial_{x_i}.
\end{equation}
We require the coefficients to be H\"older continuous, i.e.,
\begin{equation}\label{eq:condition_on_differential_op_1}
a_{ij},b_i\in\cC^{0,\tau}(\overline{Q})\text{ for some }\tau\in(0,1)\text{ for all }1\leq i,j\leq n,
\end{equation}
and assume that
\begin{equation}\label{eq:condition_on_differential_op_2}
\text{$L$ is elliptic},
\end{equation}
which means that for any $x\in\overline{Q}$ the matrix $A(x):=[a_{ij}(x)]_{i,j=1}^n\in\R^{n\times n}$ is symmetric and positive definite. 
As a last condition we require the boundary data $g$ to be independent of $t$ and with H\"older continuous second derivatives, i.e.,
\begin{equation}\label{eq:condition_on_differential_op_3}
g\in\cC^{2,\tau}(\overline{Q})\text{ for some }\tau\in(0,1).
\end{equation}

Under these conditions ``The Main Theorem'' in Ak\^o \cite{ako_dirichlet_1961} for a simplified setting can be formulated as follows.
\begin{theorem}[Ak\^o {\cite[p. 52]{ako_dirichlet_1961}}]\label{thm:ako}
Let $\hat{\rho}:\overline{Q}\times\R\rightarrow\R$ be continuous and of class $\cC^{0,\tau}$ for some $\tau\in(0,1)$ on compact subsets of $\overline{Q}\times\R$. Consider the semilinear problem
\begin{equation}\label{eq:ako_poisson_problem}
-L\varphi=\hat{\rho}(\cdot,\varphi)\text{ in }Q,\quad \varphi=g\text{ on }\partial Q
\end{equation}
with $L$ of the form \eqref{eq:form_of_L} satisfying \eqref{eq:condition_on_differential_op_1}, \eqref{eq:condition_on_differential_op_2} and $g$ satisfying \eqref{eq:condition_on_differential_op_3}. If there exist $\overline{\varphi},\underline{\varphi}\in\cC^2(\overline{Q})$ such that $\underline{\varphi}(x)\leq \overline{\varphi}(x)$ for all $x\in\overline{Q}$ and
\begin{align*}
-L\overline{\varphi}\geq\hat{\rho}(\cdot,\overline{\varphi})\text{ in }Q,\quad \overline{\varphi}\geq g\text{ on }\partial Q,\quad
-L\underline{\varphi}\leq \hat{\rho}(\cdot,\underline{\varphi})\text{ in }Q,\quad \underline{\varphi}\leq g\text{ on }\partial Q,
\end{align*}
then \eqref{eq:ako_poisson_problem} has a solution $\varphi\in\cC^2(\overline{Q})$ with $\underline{\varphi}(x)\leq \varphi(x)\leq \overline{\varphi}(x)$ for all $x\in \overline{Q}$.
\end{theorem}

The function $\overline{\varphi}$, $\underline{\varphi}$ resp., is what is called a supersolution, subsolution resp., of \eqref{eq:ako_poisson_problem}. Note that with respect to the original formulation in \cite{ako_dirichlet_1961} we do not need Ak\^o's ``$L$ being of complete type condition''. For the convenience of the reader we have added a corresponding proof of Theorem \ref{thm:ako} in Appendix \ref{sec:appendix}.

We are now ready to state our result for the general system \eqref{eq:generalized_vlasov_equation}, \eqref{eq:generalized_poisson_equation}, \eqref{eq:definition_of_spatial_densities} concerning the existence of a wide class of stationary solutions with in general non-trivial electric potential $\varphi$.

\begin{theorem}\label{thm:abstract_theorem}
Assume that the conditions \eqref{eq:condition_on_F}--\eqref{eq:condition_on_differential_op_2}, \eqref{eq:condition_on_differential_op_3} are satisfied. For every $\beta=1,\ldots,N$, let $l_\beta
\in\N\cup\{0\}$, $E_0^\beta\in\R$ and $\psi^\beta\in\cC^1(\R\times\R^{l_\beta})$, $\psi^\beta\geq 0$, such that
\begin{equation}\label{eq:cutoff_condition}
\psi^\beta(E,I)=0 \text{ for all }(E,I)\in\R\times\R^{l_\beta}\text{ with }E\geq E_0^\beta.
\end{equation}
Furthermore, for each $\beta$, let $I_1^\beta,\ldots,I_{l_\beta}^\beta\in\cI^\beta$ be a collection of $l_\beta$ first integrals. Then there exists a stationary solution $(f^1,\ldots,f^N,\varphi)$ in the sense of Definition
\ref{def:stationary_solution}, such that
\begin{enumerate}[(i)]
\item\label{item:f_beta} $f^\beta$ is given by
\[
f^\beta(x,v)=\psi^\beta\left(E^\beta(v,\varphi(x)),I_1^\beta(x,v,\varphi(x)),\ldots,I_{l_\beta}^\beta(x,v,\varphi(x))\right).
\]
\item\label{item:bounds_on_potential} The potential $\varphi$ satisfies $\underline{c}\leq \varphi(x)\leq \overline{c}$ for all $x\in\overline{Q}$, where $\underline{c},\overline{c}$ are the constants
\begin{align*}
\underline{c}&:=\min\set{\min_{x\in\partial Q}g(x),\min\set{q_\beta^{-1}E_0^\beta:q_\beta<0}},\\
\overline{c}&:=\max\set{\max_{x\in\partial Q}g(x),\max\set{q_\beta^{-1}E_0^\beta:q_\beta>0}}.
\end{align*}
\item\label{item:velocity_support} For $\beta=1,\ldots,N$, define $R^\beta:\R\rightarrow[0,\infty)$,
$
R^\beta(u)=\sqrt{2\big(E^\beta_0-q_\beta u\big)_+/m_\beta},
$
where $t_+:=\max\set{t,0}$, $t\in\R$.
Then $f^\beta(x,v)=0$, $(x,v)\in\overline{Q}\times\R^m$ provided $\abs{v}\geq R^\beta(\varphi(x))$. In particular
\begin{align*}
\supp f^\beta\subset \overline{Q}\times \overline{B_{R^\beta(\underline{c})}(0)}\text{ for all }\beta\text{ with }q_\beta>0,\\
\supp f^\beta\subset \overline{Q}\times \overline{B_{R^\beta(\overline{c})}(0)}\text{ for all }\beta\text{ with }q_\beta<0.
\end{align*}
\end{enumerate}
\end{theorem}
\begin{proof}
Let $\psi^\beta$ and all the $I^\beta_i\in\cI^\beta$ be as stated and define $\hat{f}^\beta:\overline{Q}\times\R^m\times\R\rightarrow\R$,
\[
\hat{f}^\beta(x,v,u)=\psi^\beta\left(E^\beta(v,u),I^\beta_1(x,v,u),\ldots,I^\beta_{l_\beta}(x,v,u)\right).
\]
Then, by definition of the sets $\cI^\beta$ and due to Lemma \ref{lem:energy_is_first_integral}, we have that for any $\varphi\in\cC^2(\overline{Q})$ the function $f^\beta_\varphi:\overline{Q}\times\R^m\rightarrow [0,\infty)$, $f^\beta_\varphi(x,v)=\hat{f}^\beta(x,v,\varphi(x))$ is a stationary solution to the Vlasov equation \eqref{eq:generalized_vlasov_equation} considered with that $\varphi$. I.e., there holds
\[
\ska{M^\beta v,\nabla_x f^\beta_\varphi}_{\R^n}-\frac{q_\beta}{m_\beta}\ska{\nabla_x\varphi,M^\beta\nabla_vf^\beta_\varphi}_{\R^n}+\ska{F^\beta,\nabla_vf^\beta_\varphi}_{\R^m}=0
\]
on $\overline{Q}\times\R^m\times[0,T)$ and for every $\varphi\in\cC^2(\overline{Q})$. Recall that $M^\beta$ and $F^\beta$ are allowed to depend on $t\in[0,T)$.

It therefore remains to solve the generalized Poisson equation \eqref{eq:generalized_poisson_equation}. With our ansatz for $f^\beta_\varphi$ this equation becomes the following semilinear boundary value problem
\begin{equation}\label{eq:nonlinear_poisson_equation}
-L\varphi=4\pi\sum_{\beta=1}^Nq_\beta\hat{\rho}^\beta(\cdot,\varphi)\text{ in Q},\quad \varphi=g\text{ on }\partial Q,
\end{equation}
where $\hat{\rho}^\beta:\overline{Q}\times\R\rightarrow\R$, $\hat{\rho}^\beta(x,u)=\int_{\R^m}\hat{f}^\beta(x,v,u)\:dv$.

Let us first show that $\hat{\rho}^\beta$ is well-defined and of class $\cC^1$. By the cutoff condition \eqref{eq:cutoff_condition} we have that $\hat{f}^\beta(x,v,u)=0$ for $E^\beta(x,v,u)=\frac{1}{2}m_\beta\abs{v}^2+q_\beta u\geq E^\beta_0$. Thus for fixed $x\in\overline{Q}$, $u\in\R$ we only integrate in the definition of $\hat{\rho}^\beta(x,u)$ over the open ball $B_{R^\beta(u)}(0)\subset\R^m$ with radius $R^\beta(u)$ defined in \ref{item:velocity_support}. The continuous differentiability of $\hat{\rho}^\beta$ follows via Lebesgue's dominated convergence theorem from the fact that $\hat{f}^\beta$ is $\cC^1$ on all of $\overline{Q}\times\R^m\times\R$ and that $\hat{f}^\beta$ as well as all first order partial derivatives of $\hat{f}^\beta$ are bounded on subsets with bounded $u$-component.

In order to use Theorem \ref{thm:ako} with
$
\hat{\rho}:=4\pi\sum_{\beta=1}^Nq_\beta\hat{\rho}^\beta
$
as the right-hand side in \eqref{eq:ako_poisson_problem}, it therefore only remains to find a suitable sub-/supersolution pair. As the notation already suggests this will be the constants $\underline{c},\overline{c}\in\R$ defined in \ref{item:bounds_on_potential}.

Clearly we have
$
\underline{c}\leq g(x)\leq \overline{c}
$ for all $x\in\partial Q$. Moreover, the cutoff condition \eqref{eq:cutoff_condition}, together with the definition of $\underline{c},\overline{c}$, implies
\[
\hat{\rho}(x,\overline{c})=4\pi\sum_{\set{\beta:q_\beta<0}}q_\beta\hat{\rho}^\beta(x,\overline{c})\leq 0=(-L\overline{c})(x)
\]
for every $x\in\overline{Q}$. Thus $\overline{c}$ is a supersolution. Similarly one sees that $\hat{\rho}(\cdot,\underline{c})\geq 0$ and hence $\underline{c}$ is a subsolution. Therefore Theorem \ref{thm:ako} provides us with $\varphi\in\cC^2(\overline{Q})$ solving \eqref{eq:nonlinear_poisson_equation} and satisfying \ref{item:bounds_on_potential}. Clearly $f^\beta$ now is defined as $\hat{f}^\beta(\cdot,\cdot,\varphi)$, such that $(f^1,\ldots,f^N,\varphi)$ is the desired stationary solution. Property \ref{item:f_beta} holds by definition.

Concerning property \ref{item:velocity_support} we have already seen before that $\hat{f}^\beta(x,v,u)=0$ for $\abs{v}\geq R^\beta(u)$. Hence $f^\beta(x,v)=0$ for $\abs{v}\geq R^\beta(\varphi(x))$, which in the case $q_\beta>0$ is satisfied when $\abs{v}\geq R^\beta(\underline{c})$, whereas in the case $q_\beta<0$, the condition $\abs{v}\geq R^\beta(\overline{c})$ is sufficient.
\end{proof}

\subsection{A refinement including symmetries}\label{sec:including_symmetries}
In some situations, see Sections \ref{sec:infinite_cylinder}, \ref{sec:finite_cylinder}, a function $I^\beta(x,v,\varphi(x))$ is a first integral of the characteristic system \eqref{eq:generalized_characteristic_system} when considered with potentials $\varphi$ having a specific symmetry. We still can find a stationary solution in that situation provided the semilinear problem \eqref{eq:nonlinear_poisson_equation} allows such symmetric solutions to exist. In order to select a symmetric solution we use the following ``Envelope Theorem''.
\begin{theorem}[Ak\^o {\cite[p. 55]{ako_dirichlet_1961}}]\label{thm:envelope_theorem}
In the situation of Theorem \ref{thm:ako} let $\Phi$ be the set of all $\varphi\in\cC^2(\overline{Q})$ solving \eqref{eq:ako_poisson_problem} and satisfying $\underline{\varphi}(x)\leq \varphi(x)\leq \overline{\varphi}(x)$, $x\in\overline{Q}$. Then also the functions
\[
\varphi_{sup}(x):=\sup_{\varphi\in\Phi}\varphi(x),\quad\varphi_{inf}(x):=\inf_{\varphi\in\Phi}\varphi(x)
\]
belong to $\Phi$.
\end{theorem}
A proof of the ``Envelope Theorem'' in this formulation is indicated in Remark \ref{rem:envelope_proof}.

Let now $G$ be a group acting on $\R^n$. Denote the action by $\theta*x$, $\theta\in G$, $x\in \R^n$. We assume that for all $\theta\in G$ and $\varphi\in\cC^2(\overline{Q})$ there holds 
\begin{equation}\label{eq:symmetry_assumptions}
\theta*Q=Q,\quad g(\theta*\cdot)=g,\quad L(\varphi(\theta*\cdot))=(L\varphi)(\theta*\cdot).
\end{equation}
Moreover, let $\cC^2_{sym}(\overline{Q}):=\set{\varphi\in\cC^2(\overline{Q}):\varphi(\theta*\cdot)=\varphi}$ and define in analogy to Definition \ref{def:set_of_all_integrals} $\cI^\beta_{sym}$ to be the set of all functions $I^\beta:\overline{Q}\times\R^m\times\R\rightarrow \R$, such that for all $\varphi\in\cC^2_{sym}(\overline{Q})$ the map $(x,v)\mapsto I^\beta(x,v,\varphi(x))$ is a first integral of \eqref{eq:generalized_characteristic_system} considered with that $\varphi$.

\begin{theorem}\label{thm:general_theorem_with_symmtersi}
Given the symmetry assumptions \eqref{eq:symmetry_assumptions} the statement of Theorem \ref{thm:abstract_theorem} remains valid when we replace $\cI^\beta$ by $\cI^\beta_{sym}$ provided the functions
\[
\hat{\rho}^\beta(x,u)=\int_{\R^m}\psi^\beta\left(E^\beta(v,u),I^\beta_1(x,v,u),\ldots,I^\beta_{l_\beta}(x,v,u)\right)\:dv,
\]
$\beta=1,\ldots,N$ are $G$-invariant, i.e. $\hat{\rho}^\beta(\theta*x,u)=\hat{\rho}^\beta(x,u)$, $x\in\overline{Q}$, $u\in\R$, $\theta\in G$.
\end{theorem}
\begin{proof}
From the proof of Theorem \ref{thm:abstract_theorem} we know that the boundary value problem \eqref{eq:nonlinear_poisson_equation} has a $\cC^2$ solution $\varphi$ satisfying $\underline{c}\leq \varphi(x)\leq \overline{c}$, $x\in\overline{Q}$. But since $\cI^\beta$ now has been replaced by $\cI^\beta_{sym}$ we only know that the functions $f^\beta(x,v)=\hat{f}^\beta(x,v,\varphi(x))$ satisfy the Vlasov equations \eqref{eq:generalized_vlasov_equation} when $\varphi\in\cC^2_{sym}(\overline{Q})$. Thus we need to make sure that we can find a $G$-invariant solution of \eqref{eq:nonlinear_poisson_equation}. As a consequence of our assumptions observe that with $\varphi$ also all the functions $\varphi_\theta:=\varphi(\theta*\cdot)$, $\theta\in G$ solve \eqref{eq:nonlinear_poisson_equation} and satisfy $\underline{c}\leq \varphi_\theta\leq \overline{c}$. It is then easy to see that the maximal solution $\varphi_{sup}$ of \eqref{eq:nonlinear_poisson_equation} provided by Theorem \ref{thm:envelope_theorem} has to be $G$-invariant.
\end{proof}

\section{Further properties}\label{sec:further_properties}
In this section we still consider the general Vlasov-Poisson type system \eqref{eq:generalized_vlasov_equation}, \eqref{eq:generalized_poisson_equation}, \eqref{eq:definition_of_spatial_densities} with a differential operator $L$ satisfying \eqref{eq:form_of_L}, \eqref{eq:condition_on_differential_op_1}, \eqref{eq:condition_on_differential_op_2}, but in the case of a two-component plasma, i.e. we assume $N=2$, $\beta\in\set{+,-}$ and $q_-<0<q_+$. Also we impose vanishing boundary data for the potential $\varphi$, that is $g\equiv 0$, and assume that the force term $F^\beta$ satisfying \eqref{eq:condition_on_F} depends on a positive parameter $b>0$, which in our applications will represent the strength of the external magnetic field. Moreover, throughout this section we assume for $\beta=\pm$ that the parameter dependent Vlasov equation
\begin{equation}\label{eq:generalized_vlasov_equation_with_parameter}
\partial_tf+\ska{M^\beta v,\nabla_xf^\beta}_{\R^n}-\frac{q_\beta}{m_\beta}\ska{\nabla_x\varphi,M^\beta \nabla_v f^\beta}_{\R^n}+\ska{F_b^\beta,\nabla_vf^\beta}_{\R^m}=0
\end{equation}
has a first integral $I_b^\beta\in\cI^\beta_b$ of the form
\begin{align}\label{eq:form_of_additional_integral}
I_b^\beta(x,v,u)=I_b^\beta(x,v)=\frac{b}{2}\abs{x-x_0}^2+\frac{m_\beta}{q_\beta}v^T\left(A_0(x-x_0)+a_0\right),
\end{align}
for any $b>0$, where $A_0\in\R^{m\times n}$, $a_0\in\R^m$, $x_0\in Q$ are fixed. Observe that we do not specify the dependence on the parameter $b$ in the force term $F^\beta_b$ any further, but of course a first integral as in \eqref{eq:form_of_additional_integral} can only exist for certain $F^\beta_b$. For example it is well known (Noether's Theorem, see Section 8.4 in \cite{MeyerHallOffin} for instance) that the presence of symmetries leads to the existence of first integrals. Also the existence of the specific first integrals which we will use in Sections \ref{sec:toroidal_sym}, \ref{sec:infinite_cylinder}, \ref{sec:finite_cylinder} and which will be of the general form \eqref{eq:form_of_additional_integral}, are induced by a rotational symmetry of the domain and the corresponding symmetry of the used magnetic field. Note however that in our concrete  applications we prefer to directly verify that a certain function is a first integral 
of the characteristic system instead of applying Noether's Theorem.

Under the stated assumptions Theorem \ref{thm:abstract_theorem} implies the existence of stationary solutions $(f^+,f^-,\varphi)$ of \eqref{eq:generalized_vlasov_equation_with_parameter}, \eqref{eq:generalized_poisson_equation}, \eqref{eq:definition_of_spatial_densities} with $f^\pm(x,v)=\psi^\pm(E^\pm(v,\varphi(x)),I_b^\pm(x,v))$ provided $\psi^\pm$ satisfies the cutoff condition \eqref{eq:cutoff_condition}. For simplicity we consider in this section only stationary solutions depending on the two integrals $E^\beta,I_b^\beta$.
\subsection{Confinement in space}\label{sec:confinement_in_space}

Property \ref{item:velocity_support} of Theorem \ref{thm:abstract_theorem} shows that in the considered case $g\equiv 0$ the velocity support of the solutions can be controlled in terms of the cutoff parameters $E_0^\pm$.
More precisely we have $f^\pm(x,v)=0$ for $\abs{v}\geq R_0^\pm$, where
\begin{align}\label{eq:velocity_radius}
R_0^+:=\sqrt{\frac{2}{m_+}\left(E_0^+-\frac{q_+}{q_-}E_0^-\right)_+},\quad R_0^-:=\sqrt{\frac{2}{m_-}\left(E_0^--\frac{q_-}{q_+}E_0^+\right)_+}
\end{align}
The additional integral $I_b^\beta$ allows us to control the support in the spatial dimensions as Lemma \ref{lem:confinement_in_space} shows.

\begin{lemma}\label{lem:confinement_in_space} For $\beta=\pm$, let $\psi^\beta\in\cC^1(\R^2)$, $\psi^\beta\geq 0$, satisfy \eqref{eq:cutoff_condition} and in addition
\begin{equation}\label{eq:cutoff_condition_2}
\psi^\beta(E,I)=0 \text{ for all }(E,I)\in\R^2\text{ with }I\geq I_0^\beta
\end{equation}
for some constant $I_0^\beta\in\R$. In the setting of Section \ref{sec:further_properties} any stationary solution $(f^+,f^-,\varphi)$ of \eqref{eq:generalized_vlasov_equation_with_parameter}, \eqref{eq:generalized_poisson_equation}, \eqref{eq:definition_of_spatial_densities} provided by Theorem \ref{thm:abstract_theorem} satisfies $f^\beta(x,v)=0$ for $\abs{x-x_0}\geq S_0^\beta$, where
\begin{align}\label{eq:spatial_radius}
S_0^\beta:=\frac{R_0^\beta\abs{A_0}m_\beta}{b\abs{q_\beta}}+\sqrt{\left(\frac{R_0^\beta\abs{A_0}m_\beta}{b\abs{q_\beta}}\right)^2+\frac{2\left(R_0^\beta\abs{a_0}m_\beta+\abs{q_\beta}I_0^\beta\right)}{b\abs{q_\beta}}}.
\end{align}
\end{lemma}
\begin{proof}
By \eqref{eq:cutoff_condition_2} and the form of $f^\beta$ we know that $f^\beta(x,v)=0$ provided
\[
I_b^\beta(x,v)=\frac{b}{2}\abs{x-x_0}^2+\frac{m_\beta}{q_\beta}v^T\left(A_0(x-x_0)+a_0\right)\geq I_0^\beta.
\]
Due to the bounds on the velocity support we can without restriction assume that $\abs{v}\leq R_0^\beta$. The statement then follows in a straightforward way.
\end{proof}
Also other first integrals, which depend on $\abs{x-x_0}$ in a similar way, will lead to spatially confined solutions as well. However, stating the integral $I_b^\beta$ in \eqref{eq:form_of_additional_integral} and therefore the threshold radius $S_0^\beta$ in \eqref{eq:spatial_radius} explicitly we like to emphasize that there are two ways to control the spatial supports. If we consider $S_0^\beta$ as a function of the parameter $b>0$, Lemma \ref{lem:confinement_in_space} shows that $S_0^\beta\rightarrow 0$ as $b\rightarrow\infty$. Hence a sufficiently strong magnetic field yields solutions with small spatial support. On the other hand if we consider a fixed magnetic field, i.e. a fixed $b>0$, the expressions \eqref{eq:velocity_radius}, \eqref{eq:spatial_radius} show that we can also control $S_0^\beta$ by choosing the cutoff parameters $E_0^\beta, I_0^\beta$ sufficiently small.  For clarification we like to point out that this way of controlling the spatial supports has to be understood in a purely mathematical sense, the cutoff parameters $E_0^\beta,I_0^\beta$ do not correspond to a physical controllable quantity. That said, we still conclude with the mathematical observation that a strong magnetic field is not needed for the existence of confined stationary solutions.

\subsection{Scaling of the total charges}\label{sec:scaling_of_total_charges}
The total charge of the $\beta$th component of a stationary solution $(f^+,f^-,\varphi)$ is defined by
\begin{equation}\label{eq:definition_of_total_charge}
\cQ^\beta:=q_\beta\norm{f^\beta}_{L^1(Q\times\R^m)}.
\end{equation}
It turns out that  one can increase the total charge of both components, while keeping the same spatial confinement, simply by choosing a stronger magnetic field:  
\begin{lemma}\label{lem:scaling_lemma}
Let $(f^+,f^-,\varphi)$ be a stationary solution as in Lemma \ref{lem:confinement_in_space} of the system \eqref{eq:generalized_vlasov_equation_with_parameter}, \eqref{eq:generalized_poisson_equation}, \eqref{eq:definition_of_spatial_densities} considered with parameter $b>0$. Let $R_0^\pm$, $S_0^\pm$ be the associated radii defined in \eqref{eq:velocity_radius}, \eqref{eq:spatial_radius} controlling the supports via the cutoff parameters $E_0^\pm,I_0^\pm$ and let $\cQ^\pm$ be the associated total charges. Then for any $\lambda>0$ there exists a stationary solution $(f_\lambda^+,f_\lambda^-,\varphi_\lambda)$ of the system \eqref{eq:generalized_vlasov_equation_with_parameter}, \eqref{eq:generalized_poisson_equation}, \eqref{eq:definition_of_spatial_densities} considered with parameter $\lambda b$, having total charges $\cQ_\lambda^\pm=\lambda^2\cQ^\pm$ and such that the support of $f^\pm_\lambda$ is contained in the closure of $B_{S_0^\pm}(x_0)\times B_{\lambda R_0^\pm}(0)$.
\end{lemma}
\begin{proof}
Recall that the components of the solution have the form
\[
f^\beta(x,v)=\psi^\beta\left(E^\beta(v,\varphi(x)),I_b^\beta(x,v)\right)
\]
where $I^\beta_b$ is given by \eqref{eq:form_of_additional_integral}. For $\lambda>0$ we set 
\begin{gather*}
\varphi_\lambda(x):=\lambda^2\varphi(x),\quad\psi_\lambda^\beta(E,I):=\lambda^{2-m}\psi^\beta\left(\lambda^{-2}E,\lambda^{-1}I\right),\\
f^\beta_\lambda(x,v):=\psi_\lambda^\beta\left(E^\beta(v,\varphi_\lambda(x)),I^\beta_{\lambda b}(x,v)\right).
\end{gather*}
Since by our general assumption in this section $I^\beta_{\lambda b}$ is a first integral to the Vlasov equation \eqref{eq:generalized_vlasov_equation_with_parameter} considered with parameter $\lambda b$, we directly see that $f^\beta_\lambda$ solves the Vlasov equation considered with potential $\varphi_\lambda$ and it remains to check the generalized Poisson equation \eqref{eq:generalized_poisson_equation}.

Since $g=0$, we have no problem with the boundary data, and by the change of coordinates $v=\lambda^{-1}w$ and the form of $E^\beta$, $I^\beta_b$ there holds
\begin{align*}
(4\pi)^{-1}(-L\varphi_\lambda)(x)&=\lambda^2 \sum_{\beta}q_\beta\int_{\R^m}\psi^\beta\left(E^\beta(v,\varphi(x)),I^\beta_b(x,v)\right)\:dv\\
&=\lambda^{2-m}\sum_\beta q_\beta\int_{\R^m}\psi^\beta\left(\lambda^{-2}E^\beta(w,\lambda^2\varphi(x)),\lambda^{-1} I_{\lambda b}^\beta(x,w)\right)\:dw\\
&=\sum_\beta q_\beta \int_{\R^m}f^\beta_\lambda(x,v)\:dv.
\end{align*}
Via the same transformation one also finds that the total charge associated to $f^\pm_\lambda$ is given by $\lambda^2\cQ^\pm$.

Concerning the radii controlling the supports of $f^\pm_\lambda$ we have that $\psi_\lambda^\beta(E,I)=0$ for $E\geq \lambda^2 E_0^\beta$ or $I\geq \lambda I_0^\beta$. Therefore \eqref{eq:velocity_radius} implies $f^\beta_\lambda(x,v)=0$ whenever $\abs{v}\geq \lambda R_0^\beta$, while \eqref{eq:spatial_radius} implies that the spatial cutoff radius remains at its original value $S_0^\beta$. This finishes the proof of the Lemma.
\end{proof}

\begin{remark}\label{rem:scalign_of_modified_charges} One can also easily check that besides the total charge $\cQ^\pm_\lambda$, also $\norm{\nu f^\pm_\lambda}_{L^1(Q\times\R^m)}$ with $\nu\in\cC^0(\overline{Q})$, scales quadratically with respect to $\lambda>0$.
\end{remark}

As discussed in Section \ref{sec:confinement_in_space} the strength of the magnetic field is not important for the existence of stationary confined solutions. However, Lemma \ref{lem:scaling_lemma} shows that a sufficiently strong magnetic field becomes crucial if one likes to have stationary configurations with specific prescribed total charges.

\subsection{Ratio of the total charges}\label{sec:ratio_of_charges}
By the scaling in the previous section we can increase both total charges by the same factor, which is given via the change of the parameter $b>0$.
For the construction of stationary solutions with arbitrary total charges $\cQ^+$, $\cQ^-$ it therefore remains to have for a fixed $b>0$ solutions realising an arbitrary ratio $\cQ^+/\cQ^-$. To show that this is possible we construct a family of solutions connecting the two one-component extreme cases $\cQ^+=0,\cQ^-<0$ and $\cQ^+>0,\cQ^-=0$. 

From now on let $b>0$ be fixed and abbreviate $I_b^\pm(x,v)=J\left(x,\frac{m_\beta}{q_\beta}v\right)$, 
\[
J(x,w):=\frac{b}{2}\abs{x-x_0}^2+w^T\left(A_0(x-x_0)+a_0\right).
\]
Furthermore, let $\psi\in\cC^1(\R^2)$, $\psi\geq 0$, $E_0>0$, $I_0>0$ with
\begin{gather}\label{eq:charge_ratio_condition_on_psi}
\begin{gathered}
\psi(E,I)=0 \text{ for all }(E,I)\in\R^2\text{ with }E\geq E_0\text{ or }I\geq I_0,\\
\psi\text{ bounded},\quad \psi(0,0)>0,\quad \partial_E\psi(E,I)\leq 0\text{ for all }(E,I)\in\R^2
\end{gathered}
\end{gather}
and define for $\lambda\in[0,1]$ the two functions
\begin{align}\label{eq:charge_ratio_condition_on_psi_2}
\begin{split}
\psi_\lambda^+(E,I)&:=\lambda\frac{m_+^m}{q_+^{m+1}}\psi\left(\frac{m_+}{q_+^2}E,I\right),\\
\psi_\lambda^-(E,I)&:=(1-\lambda)\frac{m_-^m}{\abs{q_-}^{m+1}}\psi\left(\frac{m_-}{q_-^2}E,I\right).
\end{split}
\end{align}
Theorem \ref{thm:abstract_theorem} and Lemma \ref{lem:confinement_in_space} tell us that for each $\lambda\in[0,1]$ there exist stationary solutions $(f_\lambda^+,f^-_\lambda,\varphi_\lambda)$ of \eqref{eq:generalized_vlasov_equation_with_parameter}, \eqref{eq:generalized_poisson_equation}, \eqref{eq:definition_of_spatial_densities} with
\begin{equation}\label{eq:form_of_family_solutions}
f^\pm_\lambda(x,v)=\psi^\pm_\lambda\left(E^\pm(v,\varphi_\lambda(x)),J\left(x,\frac{m_\beta}{q_\beta}v\right)\right),~~~ \frac{q_-}{m_-}E_0\leq \varphi_\lambda(x) \leq \frac{q_+}{m_+}E_0
\end{equation}
and with spatial and velocity supports of $f^\pm_\lambda$ controlled by $E_0$ and $I_0$.

\begin{lemma}\label{lem:continuous_family}
If in addition to \eqref{eq:form_of_L}, \eqref{eq:condition_on_differential_op_1}, \eqref{eq:condition_on_differential_op_2} there exists a nonnegative weight function $l\in\cC^0(\overline{Q})$ and $c>0$, such that 
\begin{align}\label{eq:charge_ratio_condition_on_L}
c\norm{\nabla \eta}^2_{L^2(Q)}\leq \int_Q(-L\eta)(x)\eta(x)l(x)\:dx
\end{align}
for all $\eta\in\cC^2(\overline{Q})$ with $\eta_{|\partial\Omega}=0$, then for each $\lambda\in[0,1]$ there exists  only one solution $(f^+_\lambda,f^-_\lambda,\varphi_\lambda)$ of the problem  \eqref{eq:generalized_vlasov_equation_with_parameter}, \eqref{eq:generalized_poisson_equation}, \eqref{eq:definition_of_spatial_densities} with $f^\pm=f^\pm_\lambda$ given by \eqref{eq:form_of_family_solutions}. Moreover, the associated total charges $\cQ_\lambda^\pm$ depend continuously on $\lambda\in[0,1]$ and $\cQ^+_\lambda=0$ if and only if $\lambda=0$, $\cQ^-_\lambda=0$ if and only if $\lambda=1$.
\end{lemma}
\begin{proof}
The righthand side of the semilinear Poisson equation \eqref{eq:generalized_poisson_equation} $\hat{\rho}_\lambda(x,u)$ is given by
\begin{align*}
\frac{\hat{\rho}_\lambda(x,u)}{4\pi}&=\sum_\beta q_\beta\int_{\R^m}\psi^\beta_\lambda\left(\frac{m_\beta}{2}\abs{v}^2+q_\beta u,J\left(x,\frac{m_\beta}{q_\beta}v\right)\right)\:dv\\
&=\lambda\int_{\R^m}\psi\left(\frac{1}{2}\abs{w}^2+\frac{m_+}{q_+}u,J(x,w)\right)\:dw\\
&\hspace{110pt}-(1-\lambda)\int_{\R^m}\psi\left(\frac{1}{2}\abs{w}^2+\frac{m_-}{q_-}u,J(x,w)\right)\:dw,
\end{align*}
where we used \eqref{eq:charge_ratio_condition_on_psi_2} and the change of coordinates $w=\frac{m_\beta}{q_\beta}v$.
In the proof of Theorem \ref{thm:abstract_theorem} we have already seen that $\hat{\rho}_\lambda\in \cC^1(\overline{Q}\times \R)$. Now \eqref{eq:charge_ratio_condition_on_psi} implies $\partial_u\hat{\rho}_\lambda(x,u)\leq 0$ for all $(\lambda,x,u)\in[0,1]\times\overline{Q}\times\R$. Observe also that $\partial_\lambda\hat{\rho}_\lambda(x,u)$ is independent of $\lambda$ and nonnegative.

Let $\varphi,\eta\in\cC^2(\overline{Q})$ be solutions of $-L\varphi=\hat{\rho}_\lambda(\cdot,\varphi)$, $\varphi_{|\partial Q}=0$ and $-L\eta=\hat{\rho}_{\lambda'}(\cdot,\eta)$, $\eta_{|\partial Q}=0$ with $\lambda,\lambda'\in[0,1]$ and $\frac{q_-}{m_-}E_0\leq \varphi,\eta\leq \frac{q_+}{m_+}E_0$. 
In view of \eqref{eq:charge_ratio_condition_on_L} and since $\partial_u\hat{\rho}_\lambda\leq 0$ we have
\begin{align*}
c\norm{\nabla\varphi-\nabla\eta}_{L^2(Q)}^2
&\leq\int_Q( \hat{\rho}_\lambda(\cdot,\varphi)-\hat{\rho}_{\lambda}(\cdot,\eta))(\varphi-\eta)l\:dx\\
&\hspace{55pt}+\int_Q( \hat{\rho}_\lambda(\cdot,\eta)-\hat{\rho}_{\lambda'}(\cdot,\eta))(\varphi-\eta)l\:dx\\
&\leq 0 +\abs{\lambda-\lambda'}\int_Q\partial_\lambda\hat{\rho}_0(\cdot,\eta)\:dx \left(\frac{q_+}{m_+}-\frac{q_-}{m_-}\right)E_0\norm{l}_{\cC^0(\overline{Q})}.
\end{align*}
Now also $\partial_\lambda\hat{\rho}_0(x,\eta(x))$, $x\in \overline{Q}$ is bounded by a constant depending only on $\norm{\psi}_{\cC^0}$, $E_0$, $m_\pm$, $q_\pm$ and the dimension $m$. We therefore conclude two things, first of all $\varphi=\eta$ if $\lambda=\lambda'$, such that we can define a unique family $\left(\varphi_\lambda\right)_{\lambda\in[0,1]}$ of solutions, and second we see that $\lambda\mapsto \varphi_\lambda$ is continuous as a mapping from $[0,1]$ into the Sobolev space $H^1_0(Q)$. This is (more than) enough to conclude the continuity of the associated total charges $\cQ^\pm_\lambda$ via dominated convergence.

Clearly we have $\cQ^+_0=0$ and $\cQ^-_1=0$. It therefore remains to show $\cQ^+_\lambda>0$ for $\lambda\in(0,1]$ and $\cQ^-_\lambda<0$ for $\lambda\in[0,1)$. Assume to the contrary that also $\cQ^-_\lambda=0$ for some $\lambda\in [0,1)$, which implies  
\begin{equation}\label{eq:implication_of_Q_lambda_0}
\psi\left(\frac{1}{2}\abs{w}^2+\frac{m_-}{q_-}\varphi_\lambda(x),J(x,w)\right)=0
\end{equation}
for any $(x,w)\in\overline{Q}\times \R^m$ and therefore 
\[
(-L\varphi_\lambda)(x)=4\pi \lambda\int_{\R^m}\psi\left(\frac{1}{2}\abs{w}^2+\frac{m_+}{q_+}\varphi_\lambda(x),J(x,w)\right)\:dw\geq 0,
\]
for any $x\in Q$. Recall that $\varphi_{\lambda|\partial Q}=0$. Hence the weak maximum principle implies $\varphi_\lambda\geq 0$ on all of $\overline{Q}$.
On the other hand equation \eqref{eq:implication_of_Q_lambda_0} evaluated at $(x,w)=(x_0,0)$ and condition \eqref{eq:charge_ratio_condition_on_psi} yield $\frac{m_-}{q_-}\varphi_\lambda(x_0)>0$ in contradiction to $\varphi_\lambda\geq 0$. Similarly we see that $\cQ^+_\lambda>0$ for $\lambda\in(0,1]$.
\end{proof}
\begin{remark}\label{rem:continuity_of_modified_charges} It also follows from the proof that 
if we multiply the densities $f^\pm_\lambda$ by a continuous, positive function $\nu:\overline{Q}\rightarrow(0,\infty)$, then the $L^1$-norm of $\nu f^\pm_\lambda$ has the same continuity and (non)vanishing properties as $\abs{\cQ^\pm_\lambda}$.
\end{remark}

\begin{remark}\label{rem:trivial_potential_case}
The family of solutions $(f^+_\lambda,f^-_\lambda,\varphi_\lambda)$ passes at $\lambda=\frac{1}{2}$ through a solution with trivial potential, i.e., $\varphi_{\frac{1}{2}}=0$. This follows from $\hat{\rho}_{\frac{1}{2}}(x,0)=0$ and the unique solvability.
In general, stationary solutions of the two-component Vlasov-Poisson equation with trivial potential can be found by relating the cutoff functions $\psi^+$, $\psi^-$ on an algebraic level similar to \eqref{eq:charge_ratio_condition_on_psi_2}. In that case the use of the sub-/supersolution method or a fixed point argument is not needed, see for example \cite{Knopf,SkubUMN,MMNP}. Moreover, in \cite{MMNP} it has been shown that the trivial potential case allows additional first integrals of the characteristic equations to exist.
\end{remark}

\section{Solutions with toroidal symmetry}\label{sec:toroidal_sym}
We consider the two-component Vlasov-Poisson system \eqref{eq:traditional_V}, \eqref{eq:traditional_P} in a Tokamak-like domain $\Omega\subset\R^3$. For now let us treat the Poisson equation \eqref{eq:traditional_P} with general boundary condition $\varphi_{|\partial\Omega}=g$. Later in the proof of Theorem \ref{thm:toroidal_solutions} we will simply go back to  $g\equiv 0$ in order to apply the results from Section \ref{sec:further_properties}. 

Recall from the introduction that we denoted by $\theta*x$ the $S^1$-action on $\R^3$ given by rotation around the $x_3$-axis by angle $\theta\in\R/2\pi\Z$ and that the domain $\Omega$ is assumed to satisfy $S^1*\Omega=\Omega$, $\overline{\Omega}\cap \set{(0,0,x_3):x_3\in\R}=\emptyset$. 

Similar as it has been done in \cite{Knopf} in the infinite cylinder case we will examine the system on a cross section of $\Omega$. Let
\[
Q:=\set{(r,z)\in(0,\infty)\times\R:(r,0,z)\in \Omega}.
\]

Let $B:\overline{\Omega}\rightarrow\R^3$ be a $S^1$-equivariant field and $g:\overline{\Omega}\rightarrow\R$ be $S^1$-invariant boundary data. We associate $\tilde{B}:\overline{Q}\rightarrow\R^3$, $\tilde{g}:\overline{Q}\rightarrow\R$,
\begin{equation}\label{eq:reduced_magnetic_field}
\tilde{B}(r,z)=B(r,0,z),\quad \tilde{g}(r,z)=g(r,0,z),
\end{equation}
such that for all $\theta\in S^1$, $(r,z)\in Q$ there holds
\[
B(\theta*(r,0,z))=\theta*\tilde{B}(r,z),\quad g(\theta*(r,0,z))=\tilde{g}(r,z).
\]

\begin{lemma}\label{lem:toroidally_reduced_system} Let $B:\overline{\Omega}\rightarrow\R^3$ be $S^1$-equivariant and $g:\overline{\Omega}\rightarrow\R$ be $S^1$-invariant.
If $\tilde{f}^\beta\in\cC^1(\overline{Q}\times\R^3)$, $\beta=1,\ldots,N$, $\tilde{\varphi}\in\cC^2(\overline{Q})$, $(r,z,w)\mapsto\tilde{f}^\beta(r,z,w)$, $(r,z)\mapsto\tilde{\varphi}(r,z)$  solve on $Q\times\R^3$ the stationary system
\begin{align*}\label{eq:toroidally_reduced_system}
\ska{\begin{pmatrix}
w_1\\
w_3
\end{pmatrix},\begin{pmatrix}
\partial_r \tilde{f}^\beta\\
\partial_z \tilde{f}^\beta
\end{pmatrix}}_{\R^2}&-\frac{q_\beta}{m_\beta}
\ska{\begin{pmatrix}
\partial_r\tilde{\varphi}\\
\partial_z\tilde{\varphi}
\end{pmatrix},\begin{pmatrix}
\partial_{w_1}\tilde{f}^\beta\\
\partial_{w_3}\tilde{f}^\beta
\end{pmatrix}}_{\R^2}\\&
+\ska{\frac{q_\beta}{cm_\beta}w\times \tilde{B}(r,z)+\frac{w_2}{r}\begin{pmatrix}
w_2\\-w_1\\0
\end{pmatrix},
\nabla_w\tilde{f}^\beta}_{\R^3}
=0,\\
-\left(r^{-1}\partial_r(r\partial_r)+\partial_z^2\right)\tilde{\varphi}&=4\pi\sum_{\beta=1}^Nq_\beta\int_{\R^3}\tilde{f}^\beta\:dw\text{ in }Q,\quad \tilde{\varphi}=\tilde{g}\text{ on }\partial Q,
\end{align*}
then $\left(f^1,\ldots,f^N,\varphi\right)$ defined by $f^\beta:\overline{\Omega}\times\R^3\rightarrow\R$, $\varphi:\overline{\Omega}\rightarrow\R$,
\begin{align*}
f^\beta(\theta*(r,0,z),\theta*w)&:=\tilde{f}^\beta(r,z,w),\\
\varphi(\theta*(r,0,z))&:=\tilde{\varphi}(r,z)
\end{align*}
is a stationary $S^1$-invariant solution of the original Vlasov-Poisson system \eqref{eq:traditional_V}, \eqref{eq:traditional_P} on $\Omega$ with $\varphi_{|\partial\Omega}=g$.
\end{lemma}
\begin{remark}
The transformation between solutions of the reduced system and $S^1$-invariant solutions of the full system also applies to time dependent solutions of the corresponding nonstationary equations.
\end{remark}
\begin{proof}[Proof of Lemma \ref{lem:toroidally_reduced_system}] Let us first consider the Poisson equation.
The Laplace operator in cylindrical coordinates $(r,\theta,z)$ is given by $\Delta=r^{-1}\partial_r(r\partial_r)+r^{-2}\partial_\theta+\partial_z^2$. For any $x=\theta*(r,0,z)\in\Omega$ we therefore obtain
\begin{align*}
-\Delta\varphi(\theta*(r,0,z))&=-(r^{-1}\partial_r(r\partial_r)+\partial_z^2)\tilde{\varphi}(r,z)=4\pi\sum_{\beta=1}^Nq_\beta\int_{\R^3}\tilde{f}^\beta(r,z,w)\:dw\\
&=4\pi\sum_{\beta=1}^Nq_\beta\int_{\R^3}f^\beta(\theta*(r,0,z),\theta*w)\:dw\\
&=4\pi\sum_{\beta=1}^Nq_\beta\int_{\R^3}f^\beta(\theta*(r,0,z),v)\:dv,
\end{align*}
and for $x=\theta*(r,0,z)\in \partial\Omega$, $(r,z)\in\partial Q$ we clearly have
\[
\varphi(\theta*(r,0,z))=\tilde{\varphi}(r,z)=\tilde{g}(r,z)=g(\theta*(r,0,z)).
\]

Next we turn to the Vlasov equations. Similar to before we write the arbitrary point $(x,v)\in\Omega\times\R^3$ as $x=\theta*(r,0,z)$ and $v=\theta*w$, $(r,z)\in Q$, $w\in\R^3$. The $S^1$-invariance of the functions $f^\beta$, $\varphi$ implies the equivariance for their gradients. I.e., there holds
\begin{align*}
(\nabla_xf^\beta)(\theta*(r,0,z),\theta*w)&=\theta*(\nabla_xf^\beta)(r,0,z,w),\\
(\nabla_vf^\beta)(\theta*(r,0,z),\theta*w)&=\theta*(\nabla_vf^\beta)(r,0,z,w),\\
(\nabla_x\varphi)(\theta*(r,0,z))&=\theta*(\nabla_x\varphi)(r,0,z).
\end{align*}
Moreover, we have the identities
\begin{gather*}
(\partial_{x_1}f^\beta)(r,0,z,w)=\partial_r\tilde{f}^\beta(r,z,w),\quad (\partial_{x_3}f^\beta)(r,0,z,w)=\partial_z\tilde{f}^\beta(r,z,w),\\
(\partial_{x_1}\varphi)(r,0,z)=\partial_r\tilde{\varphi}(r,z),\quad (\partial_{x_3}\varphi)(r,0,z)=\partial_z\tilde{\varphi}^\beta(r,z),\\
(\nabla_vf^\beta)(r,0,z,w)=(\nabla_w\tilde{f}^\beta)(r,z,w)
\end{gather*}
and differentiation of $f^\beta(\theta*(r,0,z),\theta*w)=\tilde{f}^\beta(r,z,w)$, $\varphi(\theta*(r,0,z))=\tilde{\varphi}(r,z)$ resp., with respect to $\theta$ at $\theta=0$ shows that
\[
(\partial_{x_2}f^\beta)(r,0,z,w)=\frac{1}{r}\ska{(\nabla_w\tilde{f}^\beta)(r,z,w),\begin{pmatrix}
w_2\\-w_1\\0
\end{pmatrix}}_{\R^3},\quad (\partial_{x_2}\varphi)(r,0,z)=0.
\]

By the invariance of the scalar product and the equivariance of the crossproduct under rotations, we therefore conclude
\begin{align*}
&\ska{\theta*w,(\nabla_xf^\beta)(\theta*(r,0,z),\theta*w)}_{\R^3}=\ska{w,(\nabla_xf^\beta)(r,0,z,w)}_{\R^3}\\
&\hspace{40pt}=\ska{\begin{pmatrix}
w_1\\w_3
\end{pmatrix},\begin{pmatrix}
\partial_r\tilde{f}^\beta(r,z,w)\\
\partial_z\tilde{f}^\beta(r,z,w)
\end{pmatrix}}_{\R^2}+\frac{w_2}{r}\ska{\begin{pmatrix}
w_2\\
-w_1\\
0
\end{pmatrix},\nabla_w\tilde{f}^\beta(r,z,w)}_{\R^3},
\end{align*}
as well as
\begin{align*}
\frac{q_\beta}{m_\beta}&\ska{-\nabla_x\varphi(\theta*(r,0,z))+c^{-1}(\theta*w)\times B(\theta*(r,0,z)),\nabla_vf^\beta(\theta*(r,0,z),\theta*w)}_{\R^3}\\
&=-\frac{q_\beta}{m_\beta}
\ska{\begin{pmatrix}
\partial_r\tilde{\varphi}(r,z)\\
\partial_z\tilde{\varphi}(r,z)
\end{pmatrix},\begin{pmatrix}
\partial_{w_1}\tilde{f}^\beta(r,z,w)\\
\partial_{w_3}\tilde{f}^\beta(r,z,w)
\end{pmatrix}}_{\R^2}\\
&\hspace{180pt}
+\frac{q_\beta}{cm_\beta}\ska{w\times \tilde{B}(r,z),
\nabla_w\tilde{f}^\beta(r,z,w)}_{\R^3}.
\end{align*}
The statement of Lemma \ref{lem:toroidally_reduced_system} follows.
\end{proof}

One immediatly sees that the reduced system in Lemma \ref{lem:toroidally_reduced_system} fits into the general setting of Section \ref{sec:general_VP} with $n=2$, $m=3$, coordinates $x=(r,z)$, $v=(w_1,w_2,w_3)$ and
\begin{gather*}\label{eq:reduced_system_fits}
M^\beta\equiv\begin{pmatrix}
1 & 0 & 0\\
0 & 0 & 1
\end{pmatrix},\quad F^\beta(r,z,w)=\frac{q_\beta}{cm_\beta}w\times\tilde{B}(r,z)+\frac{w_2}{r}\begin{pmatrix}
w_2\\
-w_1\\
0
\end{pmatrix}\\
L=\partial_r^2+\partial_z^2+r^{-1}\partial_r.
\end{gather*}
Since $r^{-1}\in \cC^{0,\tau}(\overline{Q})$, $\tau\in(0,1)$, by the form of the operator $L$ and the force $F^\beta$, the conditions \eqref{eq:condition_on_F},\eqref{eq:condition_on_differential_op_1},\eqref{eq:condition_on_differential_op_2} are satisfied. Moreover, we observe that if the support of $\tilde{f}^\beta$ is contained in the set $B_\delta(r_0,z_0)\times B_R(0)\subset \R^2\times\R^3$ for some $\delta,R>0$, then
\begin{equation}\label{eq:transformation_of_the_supports}
\supp f^\beta\subset\cO_\delta(x_0)\times B_R(0)\subset\R^3\times\R^3
\end{equation}
with $x_0=(r_0,0,z_0)$ and $\cO_\delta(x_0)$ defined in \eqref{eq:definition_of_toroidal_nbhd}. Therefore, if $\tilde{g}$ satisfies \eqref{eq:condition_on_differential_op_3}, our abstract Theorem \ref{thm:abstract_theorem} provides us via Lemma \ref{lem:toroidally_reduced_system} with infinitely many stationary solutions of \eqref{eq:traditional_V}, \eqref{eq:traditional_P} in domains with toroidal symmetry. Concerning the spatial confinement it only remains to find a first integral of the associated characteristic system
\begin{align}\label{eq:characteristic_system_in_toroidal_case}
\begin{split}
\dot{r}&=w_1,\quad \dot{z}=w_3,\\
\dot{w}&=-\frac{q_\beta}{m_\beta}
\begin{pmatrix}
\partial_r\tilde{\varphi}(r,z)\\
0\\
\partial_z\tilde{\varphi}(r,z)
\end{pmatrix}+\frac{q_\beta}{cm_\beta}w\times \tilde{B}(r,z)+\frac{w_2}{r}\begin{pmatrix}
w_2\\-w_1\\0
\end{pmatrix}.
\end{split}
\end{align}
\begin{lemma}\label{lem:existence_of_additional_first_integral_in_toroidal_case}
a) A $S^1$-equivariant field $B:\overline{\Omega}\rightarrow\R^3$ is divergence-free if and only if the reduced field $\tilde{B}:\overline{Q}\rightarrow\R^3$, $\tilde{B}\in\cC^1(\overline{Q})$, defined in \eqref{eq:reduced_magnetic_field} satisfies
\begin{equation}\label{eq:divergence_free_condition}
\partial_r(r\tilde{B}_1)+\partial_z(r\tilde{B}_3)=0.
\end{equation}
b) Condition \eqref{eq:divergence_free_condition} holds true if $\tilde{B}$ has the form
\[
\tilde{B}_1=r^{-1}\partial_z\tilde{A},\quad \tilde{B}_3=-r^{-1}\partial_r\tilde{A}
\]
for some function $\tilde{A}\in\cC^2(\overline{Q})$. Moreover, in that case $I^\beta:\overline{Q}\times\R^3\rightarrow\R$,
\[
I^\beta(r,z,w)=\tilde{A}(r,z)-\frac{cm_\beta}{q_\beta}rw_2
\]
is a first integral of \eqref{eq:characteristic_system_in_toroidal_case}.
\end{lemma}
\begin{proof}
a) First of all observe that the divergence of $B$ is $S^1$-invariant. Indeed differentiation of the identity $B(\theta*x)=\theta*B(x)$ with respect to $x$ in the direction $y$ shows that the Jacobian $DB:=
\left(
\begin{array}{ccc}
\partial_{x_1}B_1 &\partial_{x_2}B_1 &\partial_{x_3}B_1   \\
\partial_{x_1}B_2 &\partial_{x_2}B_2 &\partial_{x_3}B_2   \\
\partial_{x_1}B_3 &\partial_{x_2}B_3 &\partial_{x_3}B_3   
\end{array}\right)
:\Omega\rightarrow \R^{3\times 3}$ satisfies
\begin{equation}\label{e4.5}
(DB)(\theta*x)(\theta*y)=\theta*((DB)(x)(y))
\end{equation}
for all $\theta\in S^1$, $x\in\overline{\Omega}$ and $y\in\R^3$, where in the left hand side of \eqref{e4.5} we consider the Jacobian $DB$ at the point $\theta*x$ and then multiply the matrix $(DB)(\theta*x)$ by the vector $(\theta*y)$, in the right hand side of \eqref{e4.5} we consider the Jacobian $DB$ at the point $x$, multiply the matrix $(DB)(x)$ by the vector $(y)$ and consider the rotation by the angle $\theta$. Let $\{y^{(j)}\}$, 
$j=1,2,3$, be an orthonormal basis in $\R^3$. Then $\{\theta*y^{(j)}\}$ is also 
an orthonormal basis in $\R^3$. Since a trace of martix is invariant under an orthonormal transformation, using identity \eqref{e4.5}, we obtain 
\begin{align*}
\divv &B(\theta*x)=\tr (DB)(\theta*x)=\sum_{i=1}^3\ska{(DB)(\theta*x)(\theta*y^{(i)}),(\theta*y^{(i)})}_{\R^3}\\
&=\sum_{i=1}^3\ska{\theta*\left((DB)(x)(y^{(i)})\right),\theta*y^{(i)}}_{\R^3}=\sum_{i=1}^3\ska{(DB)(x)(y^{(i)}), y^{(i)}}_{\R^3}=\divv B(x).
\end{align*}
In the same way as in the proof of Lemma \ref{lem:toroidally_reduced_system} one can see that
\begin{gather*}
(\partial_{x_1}B_1)(r,0,z)=\partial_r\tilde{B}_1(r,z),\quad(\partial_{x_3}B_3)(r,0,z)=\partial_z\tilde{B}_3(r,z),\\
(\partial_{x_2}B_2)(r,0,z)=r^{-1}\tilde{B}_1(r,z)
\end{gather*}
for $(r,z)\in Q$. For any $x=\theta*(r,0,z)\in\Omega$, we therefore conclude
\[
\divv B(\theta*(r,0,z))=\divv B(r,0,z)=\partial_r\tilde{B}_1(r,z)+r^{-1}\tilde{B}_1(r,z)+\partial_z\tilde{B}_3(r,z)
\]
and the statement follows.

b) The first part is clear. Let us now differentiate $I^\beta$ along a solution $(r,z,w)$ of the characteristic system \eqref{eq:characteristic_system_in_toroidal_case} defined on a time interval $J$ and mapping into $Q\times\R^3$. There holds
\begin{align*}
\frac{d}{dt}I^\beta(r,z,w)&=\partial_r\tilde{A}(r,z)w_1+\partial_z\tilde{A}(r,z)w_3-\frac{cm_\beta}{q_\beta}w_1w_2\\
&\hspace{80pt}-\frac{cm_\beta}{q_\beta}r\left(\frac{q_\beta}{cm_\beta}[w\times \tilde{B}(r,z)]_2-\frac{w_1w_2}{r}\right)=0,
\end{align*}
since the second component of the crossproduct is given by
\[
[w\times\tilde{B}(r,z)]_2=\frac{1}{r}(w_3\partial_z\tilde{A}(r,z)+w_1\partial_r\tilde{A}(r,z)).
\]
Thus $I^\beta$ is a first integral.
\end{proof}

The integral $I^\beta$ is induced by the rotational symmetry of the original system. It is independent of the potential $\tilde{\varphi}$ and the second component of the cross-sectional magnetic field $\tilde{B}_2$. By a suitable choice of $\tilde{A}$ it can be used to control the spatial support of our stationary solutions as we will see below. Going back to the original equivariant magnetic field $B$ defined on $\Omega=S^1*Q$, we have that the poloidal part $B_{pol}(x):=P_xB(x)$ and the toroidal part $B_{tor}(x):=(\id_{\R^3}-P_x)B(x)$, cf. Section \ref{sec:introduction} for the precise definition of $P_x$, are given by
\[
B_{pol}(\theta*(r,0,z))=\theta*\begin{pmatrix}
\tilde{B}_1(r,z)\\
0\\
\tilde{B}_3(r,z)
\end{pmatrix},\quad B_{tor}(\theta*(r,0,z))=\theta*\begin{pmatrix}
0\\
\tilde{B}_2(r,z)\\
0
\end{pmatrix}.
\]
Therefore, as already discussed in the introduction after Theorem \ref{thm:toroidal_solutions}, the toroidal part has no influence on the control of the support of the stationary solutions. Moreover, a purely toroidal magnetic field, i.e. $\tilde{A}=0$, does not provide the existence of a first integral $I^\beta$  useful for the confinement. Indeed, it is known in plasma theory \cite[Chap.3,\S 3.5]{Miyamoto}, that a purely toroidal magnetic field does not provide plasma confinement due to a drift of the charged particles.

In view of Theorem \ref{thm:toroidal_solutions} it remains to combine the results from Sections \ref{sec:general_VP}, \ref{sec:further_properties} with the just presented Lemmas.

\begin{proof}[Proof of Theorem \ref{thm:toroidal_solutions}] Let $x_0\in\Omega$, $r_0:=\sqrt{x_{0,1}^2+x_{0,2}^2}$, $z_0:=x_{0,3}$, as well as $g\equiv 0$ and $B_b:\overline{\Omega}\rightarrow\R^3$, $b>0$ be $S^1$-equivariant with poloidal part
\[
P_xB_b(x)=\frac{b}{x_1^2+x_2^2}\begin{pmatrix}
x_1(x_3-z_0)\\
x_2(x_3-z_0)\\
-\sqrt{x_1^2+x_2^2}(\sqrt{x_1^2+x_2^2}-r_0)
\end{pmatrix}.
\]
The induced cross-sectional magnetic field $\tilde{B}_b:\overline{Q}\rightarrow\R^3$ reads 
\[
\tilde{B}_b(r,z)=\frac{b}{r}\begin{pmatrix}
z-z_0\\
0\\
-(r-r_0)
\end{pmatrix}
+\begin{pmatrix}
0\\\tilde{B}_2(r,z)\\0
\end{pmatrix},
\]
where $\tilde{B}_2$ corresponds to the unspecified toroidal part of $B_b$, which may or may not depend on the parameter $b$. In view of Lemma \ref{lem:existence_of_additional_first_integral_in_toroidal_case} b) one observes that $\tilde{B}_b$ is induced by the function $\tilde{A}_b:\overline{Q}\rightarrow\R$,
$
\tilde{A}_b(r,z)=\frac{b}{2}\big((r-r_0)^2+(z-z_0)^2\big)
$, such that besides the energy $E^\pm(r,z,w)=\frac{1}{2}m_\pm\abs{w}^2+q_\pm \tilde{\varphi}(r,z)$ also 
\[
I^\pm_b(r,z,w)=\frac{b}{2}\big((r-r_0)^2+(z-z_0)^2\big)-\frac{cm_\pm}{q_\pm}rw_2
\]
is a first integral of the characteristic system \eqref{eq:characteristic_system_in_toroidal_case} of the reduced Vlasov-Poisson system stated in Lemma \ref{lem:toroidally_reduced_system}.

We have already checked, cf. page \pageref{eq:reduced_system_fits}, that this reduced system fits into the general framework of Section \ref{sec:general_VP}. Moreover, the additional integral $I^\pm_b$ clearly is of the form \eqref{eq:form_of_additional_integral}, such that also the Lemmas stated in Section \ref{sec:further_properties} for the two component system with vanishing boundary data, which we consider here, apply. For condition \eqref{eq:charge_ratio_condition_on_L} in Lemma \ref{lem:continuous_family} observe that for any $g\in\cC^2(\overline{Q})$, $g_{|\partial Q}=0$ there holds
\begin{align*}
\int_Q(-L g)(r,z)g(r,z)r\:d(r,z)&=-\int_Q\big(\partial_r(r\partial_r g(r,z))+r\partial_z^2 g(r,z)\big)g(r,z)\:d(r,z)\\
&=\int_Qr\abs{\nabla g(r,z)}^2\:d(r,z)\geq c_0\norm{\nabla g}^2_{L^2(Q)},
\end{align*}
where $\nabla=(\partial_r,\partial_z)$ and $c_0=\min\set{r:(r,z)\in\overline{Q}}>0$.

Let us now fix $b>0$ and turn to the proof of part (i) of Theorem \ref{thm:toroidal_solutions}. It follows from Theorem \ref{thm:abstract_theorem} and Lemmas \ref{lem:confinement_in_space}, \ref{lem:continuous_family} that the cross-sectional system stated in Lemma \ref{lem:toroidally_reduced_system} has a stationary solution $(\tilde{f}^+,\tilde{f}^-,\tilde{\varphi})$ with $\supp \tilde{f}^\pm$ arbitrarily small localized around $(r_0,z_0,0)\in Q\times\R^3$ and with arbitrary total charge ratio $\tilde{\cQ}^+/\tilde{\cQ}^-$, where
$\tilde{\cQ}^\pm=q_\pm||\tilde{f}^\pm||_{L^1(Q\times\R^3)}$ denote the cross-sectional total charges. Indeed we choose cutoff functions $\psi^\pm_\lambda:\R^2\rightarrow[0,\infty)$ and parameters $E_0,I_0>0$ as in \eqref{eq:charge_ratio_condition_on_psi}, \eqref{eq:charge_ratio_condition_on_psi_2}, such that 
\[
\tilde{f}^\pm(r,z,w)=\psi_\lambda^\pm\left(\frac{1}{2}m_\pm\abs{w}^2+q_\pm \tilde{\varphi}(r,z),\frac{b}{2}\big((r-r_0)^2+(z-z_0)^2\big)-\frac{cm_\pm}{q_\pm}rw_2\right).
\]
 A sufficiently small choice of $E_0$, $I_0$ allows us to control the support of the solutions via \eqref{eq:velocity_radius} and \eqref{eq:spatial_radius}, whereas the parameter $\lambda\in(0,1)$ allows us to adjust the cross-sectional charge ratio, see Lemma \ref{lem:continuous_family}. We then use Lemma \ref{lem:toroidally_reduced_system} to transform $(\tilde{f}^+,\tilde{f}^-,\tilde{\varphi})$ to a stationary solution $(f^+,f^-,\varphi)$ of the original system stated on the toroidal domain $\Omega$. The support of $f^\pm$ is obtained by applying the $S^1$-action to the support of $\tilde{f}^\pm$, cf. \eqref{eq:transformation_of_the_supports}. For the total charges there holds
\begin{align*}
\cQ^\pm&=q_\pm\int_{\Omega\times\R^3}f^\pm(x,v)\:d(x,v)=2\pi q_\pm\int_{Q\times\R^3}\tilde{f}^\pm(r,z,w)r\:d(r,z,w),
\end{align*}
which in the same way as the cross-sectional total charges can be adjusted via the parameter $\lambda\in(0,1)$, see Remark \ref{rem:continuity_of_modified_charges}. This finishes the proof of part (i).

Part (ii) is a direct consequence of Lemma \ref{lem:scaling_lemma}, Remark \ref{rem:scalign_of_modified_charges} resp.. 
\end{proof}

\section{Solutions in an infinite cylinder}\label{sec:infinite_cylinder}
In this section we consider the two-component Vlasov-Poisson system  \eqref{eq:traditional_V}, \eqref{eq:traditional_P} with boundary condition \eqref{eq:trad_BC} on an infinite cylinder
$\Omega=Q{\times}\mathbb{R}$, where $Q=B_{r_0}(0)\subset\mathbb{R}^2$ is the two-dimensional disc of radius $r_0>0$ centered at $0$.
Concerning the magnetic field we suppose that $B$ is constant and directed along the cylinder axis, i.e., 
\[
B(x)=(0,0,b),\quad x\in\overline{\Omega}
\]
for some $b>0$.

In \cite{Knopf} it has been shown that the three-dimensional Vlasov-Poisson system can be reduced to a two-dimensional system with the following ansatzes. Restricting ourselves to the stationary case we assume the functions $f^\beta$, $\varphi$ to be independent
of $x_3$ and to have the structure 
\[
f^\beta(x,v)=\bar{f}^\beta(x_1,x_2,v_1,v_2)\xi^\beta(v_3),\quad \varphi(x)=\bar{\varphi}(x_1,x_2)
\]
where 
$\bar{f}^\beta:\overline{Q}\times\R^2\rightarrow[0,\infty)$
is continuously differentiable with respect to all its variables and compactly
supported, $\xi^\beta\in\cC_0^1(\R)$ with 
$\int_{-\infty}^{\infty}\xi^\beta(s)ds=1$, $\bar{\varphi}\in\cC^2(\overline{Q})$.
Under these assumptions the reduced stationary system corresponding to \eqref{eq:traditional_V}, \eqref{eq:traditional_P}, \eqref{eq:trad_BC} then reads
\begin{align}
\begin{gathered}
\label{eq:reduced final}
\ska{v',\nabla_{x'}\bar{f}^\beta}_{\R^2}+\frac{q_\beta}{m_\beta}
\ska{-\nabla_{x'}\bar{\varphi}(x')+\frac{b}{c} v^{\bot} ,\nabla_{v'}\bar{f}^\beta}_{\R^2}=0\ \text{in }Q\times\R^2,\\
-\Delta\bar{\varphi}(x')=4\pi\sum_\beta q_\beta\int_{\R^2}\bar{f}^\beta\:dv'\ \text{in }Q,\\
\bar{\varphi}(x')=0\ \text{on }\partial Q,
\end{gathered}
\end{align}
where $x'=(x_1,x_2)^T$, $v'=(v_1,v_2)^T$, $v^{\bot}=(v_2,-v_1)^T$. Note that the reduction in \cite{Knopf} has been done in the case $\Omega=\R^3$ with a decay condition for the potential instead of boundary condition \eqref{eq:trad_BC}, but it easily carries over to our setting.

Equation \eqref{eq:reduced final} fits into the class of equations treated in Section \ref{sec:general_VP}. Since \eqref{eq:condition_on_F} holds true, the energy
\[
E^\beta(v',\bar{\varphi}(x'))=\frac{1}{2}m_{\beta}|v'|^2+q_{\beta}\bar{\varphi}(x')
\]
is a first integral of the corresponding characteristic system 
\begin{equation}\label{eq:characteristics_infinite_cylinder}
\dot{x}'=v', \quad
\dot{v}'=
\frac{q_\beta}{m_\beta}
\left(-\nabla_{x'}\bar{\varphi}(x')+\frac{b}{c}v^{\bot} \right).
\end{equation}
Moreover, it has been shown in \cite{Knopf} that if one considers \eqref{eq:characteristics_infinite_cylinder} with a potential $\bar{\varphi}$ depending only on $|x'|$, then also the function
\begin{align}\label{e5.3}
I^\beta_b(x',v')=\frac{b}{2}|x'|^2+\frac{cm_\beta}{q_\beta} \ska{x',v^{\perp}}_{\R^2}
\end{align}
is a first integral of \eqref{eq:characteristics_infinite_cylinder}. Observe that $I^\beta_b$ is of the form \eqref{eq:form_of_additional_integral}.

Now instead of reducing system \eqref{eq:reduced final} further by the use of polar coordinates, similar to the use of cylindrical coordinates in Section \ref{sec:toroidal_sym}
and also as it has been done in \cite{Knopf}, it seems here more convenient to deal with the needed symmetry by means of Theorem \ref{thm:general_theorem_with_symmtersi}.

Indeed \eqref{eq:symmetry_assumptions} holds true for $G=S^1$ acting via rotations $\theta*x'$ by angle $\theta\in \R/2\pi\Z=S^1$, $Q=B_{r_0}(0)$, $g=0$ and $L=\Delta$. Moreover, for all $\psi^\pm\in\cC^1(\R^2)$, $\psi^\pm\geq 0$ satisfying \eqref{eq:cutoff_condition} the functions 
\begin{align*}
\hat{\rho}^\pm(x',u)&=\int_{\R^2}\psi^\pm\left(E^\pm(v',u),I^\pm_b(x',v')\right)\:dv'
\end{align*}
satisfy $\hat{\rho}^\pm(\theta*x',u)=\hat{\rho}^\pm(x',u)$, $\theta\in S^1$, $x'\in Q$, $u\in\R$ 
as can be seen by the transformation $v'=\theta*w'$. In consequence of Theorem \ref{thm:general_theorem_with_symmtersi} and Lemma \ref{lem:confinement_in_space} we conclude
\begin{corollary}\label{cor:infinite_cylinder}
Let $\psi^\pm\in\cC^1(\R^2)$, $\psi^\beta\geq 0$ with $\psi^\pm(E,I)=0$ for $E\geq E_0^\pm$ or $I\geq I_0^\pm$, as well as $b>0$. There exists a stationary solution $(\bar{f}^+,\bar{f}^-,\bar{\varphi})$ of problem \eqref{eq:reduced final} satisfying
\[
\bar{f}^\pm(x',v')=\psi^\pm\left(E^\pm(v',\bar{\varphi}(x')),I^\pm_b(x',v')\right),
\]
$\bar{\varphi}(x')=\bar{\varphi}(\abs{x'})$ and $\supp \bar{f}^\pm\subset \overline{B_{S_0^\pm}(0)}\times \overline{B_{R_0^\pm}(0)}$ with $R_0^\pm$, $S_0^\pm$ given in \eqref{eq:velocity_radius}, \eqref{eq:spatial_radius}.
\end{corollary}

\begin{remark}
a) For fixed $b>0$ we therefore find stationary solutions of \eqref{eq:reduced final} whose support can be made arbitrarily sharp simply by a sufficiently small choice of the cutoff values $E_0^\pm$, $I_0^\pm$. Moreover, by choosing $\psi^\pm$ as in \eqref{eq:charge_ratio_condition_on_psi}, \eqref{eq:charge_ratio_condition_on_psi_2} we can also prescribe the ratio between the two total charges $q_+\norm{\bar{f}^+}_{L^1(Q\times\R^2)}$ and $q_-\norm{\bar{f}^-}_{L^1(Q\times\R^2)}$. As in Theorem \ref{thm:toroidal_solutions} (ii) one can then scale the total charges without changing their ratio nor the spatial supports of the solutions by increasing the strength $b$ of the magnetic field, cf. Lemma \ref{lem:scaling_lemma}.

b) Taking into account the reduction process leading to \eqref{eq:reduced final}, we see that every solution $(\bar{f}^+,\bar{f}^-,\bar{\varphi})$ as in Corollary \ref{cor:infinite_cylinder} gives rise to solutions of the original system \eqref{eq:traditional_V}, \eqref{eq:traditional_P}, \eqref{eq:trad_BC}
by setting $f^\pm(x,v)=\bar{f}^\pm(x',v')\xi^\pm(v_3)$ and $\varphi(x)=\bar{\varphi}(x')$. Thus we obtain stationary solutions whose spatial supports are contained in arbitrarily small cylinders $B_\delta(0)\times \R\subset \Omega$. Note however, that concerning the total charges it only makes sense to compare the cross-sectional charges, since $f^\pm\notin L^1(\Omega\times\R^3)$.

c) From \eqref{eq:form_of_additional_integral}, \eqref{eq:spatial_radius}, and \eqref{e5.3} it follows that $S^\beta_0\to 2r^\beta_0$ as $I^\beta_0\to 0$, where $r^\beta_0=\dfrac{R^\beta_0m_\beta c}{b|q_\beta|}$ is the Larmor radius, see \cite[Chap.2,\S 2.3]{Miyamoto}.
\end{remark}

\section{Mirror trap solutions in a finite cylinder}\label{sec:finite_cylinder}
In this last section we give examples of stationary solutions in a finite cylinder. The solutions correspond to a two-component plasma confined in a Mirror trap. Here the supports touch the boundary of the domain, but only in two small prescribed discs at the top and the bottom of the cylinder.

We consider system \eqref{eq:traditional_V}, \eqref{eq:traditional_P}, \eqref{eq:trad_BC} 
in a finite cylinder $\Omega=Q=B_{r_0}(0){\times}(-l,l)$, $B_{r_0}(0)\subset\R^2$, $r_0>0$, $l>0$.
Let the magnetic field $B(x)$ be given by
\begin{gather}\label{field_finite}
B(x)=\begin{pmatrix}
-\frac{1}{2}a'(x_3)x_1\\
-\frac{1}{2}a'(x_3)x_2\\
 a(x_3)
\end{pmatrix},
\end{gather}
where $a:[-l,l]\rightarrow (0,\infty)$ is $\cC^1$. 
If $a(x_3)\equiv a$, $a\in \R$, we get back the constant magnetic field directed along the
cylinder axis, which we have used in Section \ref{sec:infinite_cylinder}. Observe that $B$ is divergence-free.
Similar to the previous section, but now with $x_3$ dependence, we are interested in potentials of the form $\varphi(x)=\varphi(|x'|,x_3)$, where 
$x'=(x_1,x_2)\in \overline{B_{r_0}(0)}$, $x_3\in[-l,l]$.
For such $\varphi$ the function 
$$I^\beta(x,v)=\frac{a(x_3)}{2}|x'|^2+\frac{cm_\beta}{q_\beta}\ska{x',v^{\perp}}_{\R^2}=
\frac{a(x_3)}{2}(x_1^2+x_2^2)+\frac{cm_\beta}{q_\beta}(x_1v_2-x_2v_1)$$
as well as the energy 
$$E^\beta(v,\varphi(x))=\frac{1}{2}m_\beta|v|^2+q_\beta\varphi(x)$$
are  first integrals of the characteristic system
\begin{gather*}
\dot{x}=v,\quad \dot{v}=-\frac{q_\beta}{m_\beta}\nabla\varphi(x)+\frac{q_\beta}{cm_\beta}\begin{pmatrix}
a(x_3)v_2+\frac{1}{2}a'(x_3)x_2v_3\\
-a(x_3)v_1-\frac{1}{2}a'(x_3)x_1v_3\\
\frac{1}{2}a'(x_3)(x_1v_2-x_2v_1)
\end{pmatrix}
\end{gather*}
Applying Theorem \ref{thm:general_theorem_with_symmtersi}
we obtain the following stationary solutions.
\begin{corollary}
Let the magnetic field $B$ be given by \eqref{field_finite}.
There exists a stationary solution $(f^{+},f^{-},\varphi)$ for the system \eqref{eq:traditional_V}, \eqref{eq:traditional_P}, \eqref{eq:trad_BC}
considered on the cylinder $Q$ with $f^\beta=\psi^\beta(E^\beta,I^\beta)$,
whenever $\psi^\beta \in C^1(\R,\R)$ with
$\psi^\beta(E,I)=0$ for $E\geq E^\beta_0$ or
$I\geq I^\beta_0$. 
\end{corollary}
\begin{remark} a) Note that the integral $I^\beta$ is not exactly of the form \eqref{eq:form_of_additional_integral}. Instead of this for every $x_3\in[-l,l]$ in the function $(x',v)\mapsto I^\beta(x',x_3,v)$ the coefficient $a(x_3)$ takes the role of the parameter $b$.
Having a quick look at Sections \ref{sec:scaling_of_total_charges}, \ref{sec:ratio_of_charges} one can check that this still allows us in the usual way to prescribe the ratio between the total charges and to scale their absolute values. In order to not risk too much repetition we omit any further details.

b) In analogy to Lemma \ref{lem:confinement_in_space} it is also easy to see that the value of the function $a$ at height $x_3\in[-l,l]$ enters formula \eqref{eq:spatial_radius} instead of $b$.
More precisely, there holds $f^\beta(x,v)=0$ for
\[
\abs{x'}\geq \frac{R_0^\beta\abs{A_0}m_\beta}{a(x_3)\abs{q_\beta}}+\sqrt{\left(\frac{R_0^\beta\abs{A_0}m_\beta}{a(x_3)\abs{q_\beta}}\right)^2+\frac{2I_0^\beta}{a(x_3)}},
\]
where $R_0^\beta$ is defined in \eqref{eq:velocity_radius} and $A_0$ reads
\[
A_0=\begin{pmatrix}
0&-c\\
c&0\\
0&0
\end{pmatrix}.
\]
 This way contrary to Section \ref{sec:infinite_cylinder} the cross-sectional supports $\supp f^\beta\cap \R^2\times\{x_3\}\times\R^3$ now can now vary with $x_3$.
Increasing $a$ towards the boundary points $x_3=\pm l$ one can construct solutions, such that
the plasma hits the boundary of the cylinder $Q$ only in a small region 
at the top $B_{r_0}(0)\times\{l\}$ and the bottom $B_{r_0}(0)\times\{-l\}$ of the cylinder.
\end{remark}

\appendix
\section{A proof of Ak\^o's Theorem}\label{sec:appendix}
We will quickly provide a proof of Ak\^o's Theorem in the simplified formulation we use. Contrary to Ak\^o's formulation and proof of the Theorem we do not require the elliptic operator $L$ to be of ``complete type''. 
 
\begin{proof}[Proof of Theorem \ref{thm:ako}]
For $\alpha\in(0,\tau]$ let $S_\alpha:\cC^{0,\alpha}(\overline{Q})\rightarrow \cC^{2,\alpha}(\overline{Q})$, $f\mapsto S_\alpha(f)$ be the solution operator of 
\[
-L\varphi=f\text{ in }Q,\quad \varphi=g\text{ on }\partial Q.
\]
Due to Schauder theory $S_\alpha$ is well-defined, continuous and satisfies
\begin{equation}\label{eq:schauder_estimate}
\norm{S_\alpha(f)}_\hoeld{2,\alpha}\leq C\left(\norm{f}_\hoeld{0,\alpha}+\norm{g}_\hoeld{2,\alpha}\right)
\end{equation}
for a constant $C>0$ depending on $\alpha$, the dimension $n$ and the coefficients of $L$, see for example \cite{GT} Theorems 6.6, 6.14 and Corollary 3.8.

Next we define a modified right-hand side $\tilde{\rho}:\overline{Q}\times \R\rightarrow \R$,
\begin{equation}\label{eq:definition_of_rho_tilde}
\tilde{\rho}(x,r):=\hat{\rho}\left(x,\max\set{\underline{\varphi}(x),\min\set{\overline{\varphi}(x),r}}\right),
\end{equation}
which coincides with the old right-hand side $\hat{\rho}$ whenever $r\in \left[\underline{\varphi}(x),\overline{\varphi}(x)\right]$.

Clearly $\tilde{\rho}$ is bounded by $\norm{\hat{\rho}}_{\cC^0(K)}$, where $K:=\overline{Q}\times \left[\min\underline{\varphi},\max\overline{\varphi}\right]$. Moreover, observe that $\tilde{\rho}$ is globally $\tau$-H\"older continuous as a composition of a map that is $\tau$-H\"older continuous on $K$ and a Lipschitz map only taking values in $K$. 

Fix $\alpha\in (0,\tau^2)$. In view of Lemma \ref{lem:A2} below and the continuity of the solution operator $S_\alpha$ the map $F:\hoeld{0,\tau}\rightarrow \hoeld{0,\tau}$,
\[
F(\varphi):=S_\alpha(\tilde{\rho}(\cdot,\varphi))
\]
is continuous. Since it actually maps into $\hoeld{2,\alpha}$ it is also compact. Hence Schauder's fixed point theorem applies provided we find a closed ball $B\subset \hoeld{0,\tau}$ with $F(B)\subset B$. 

Let $R>0$ and $\varphi\in \hoeld{0,\tau}$ with $\norm{\varphi}_\hoeld{0,\tau}\leq R$. Using inequalities \eqref{eq:schauder_estimate} and \eqref{eq:rho_inequaltiy} from Lemma \ref{lem:A1} below one sees that
\begin{align*}
\norm{F(\varphi)}_\hoeld{0,\tau}&\leq C\norm{F(\varphi)}_\hoeld{2,\alpha}\leq C\left(\norm{\tilde{\rho}(\cdot,\varphi))}_\hoeld{0,\alpha}+\norm{g}_\hoeld{2,\alpha}\right)\\
&\leq C\left(\left(1+\norm{\varphi}_\hoeld{0,\tau}\right)^\tau+\norm{g}_\hoeld{2,\tau}\right)\leq C\left(1+R\right)^\tau+C,
\end{align*}
where the $\varphi$-independent constants $C>0$ are adapted from estimate to estimate. We conclude that $\norm{F(\varphi)}_\hoeld{0,\tau}\leq R$ for sufficiently large $R>0$.

In consequence there exists $\varphi\in \hoeld{0,\tau}$ with $\varphi=F(\varphi)$. Observe that this means that $\varphi$ is actually contained in $\hoeld{2,\alpha}$ and 
\[
-L\varphi=\tilde{\rho}(\cdot,\varphi)\text{ in }Q,\quad \varphi=g\text{ on }\partial Q.
\]

It remains to show $\underline{\varphi}(x)\leq \varphi(x)\leq \overline{\varphi}(x)$ for all $x\in\overline{Q}$. Note that then the modified right-hand side $\tilde{\rho}(\cdot,\varphi)$ actually coincides with the original right-hand side $\hat{\rho}(\cdot,\varphi)$ and we therefore have found the desired solution.

Let $Q_>:=\set{x\in Q:\varphi(x)>\overline{\varphi}(x)}$ and assume to the contrary that $Q_>$ is not empty. Observe that $\varphi_{|\partial Q}=g$ and $\overline{\varphi}_{|\partial Q}\geq g$ implies $\partial Q_>\subset \set{x\in\overline{Q}:\varphi(x)=\overline{\varphi}(x)}$. Moreover, for $x\in Q_>$ there holds
\[
(-L(\overline{\varphi}-\varphi))(x)\geq \hat{\rho}(x,\overline{\varphi}(x))-\tilde{\rho}(x,\varphi(x))=0,
\]
such that the weak maximum principle implies
\[
\inf_{Q_>}(\overline{\varphi}-\varphi)=\min_{\partial Q_>}(\overline{\varphi}-\varphi)=0,
\]
but this contradicts the definition of $Q_>$. A similar reasoning shows that also the set $Q_<:=\set{x\in Q:\varphi(x)<\underline{\varphi}(x)}$ is empty. This finishes the proof of Theorem \ref{thm:ako}.
\end{proof}

\begin{remark}\label{rem:envelope_proof}
Note that if $\big(\underline{\varphi}_j\big)_{j=1}^l$, $\left(\overline{\varphi}_j\right)_{j=1}^m$ are finite families of subsolutions, supersolutions resp., then one can obtain a solution $\varphi$ satisfying
\begin{align*}
\max\big\{\underline{\varphi}_1(x),\ldots,\underline{\varphi}_l(x)\big\}\leq \varphi(x)\leq \min \big\{\overline{\varphi}_1(x),\ldots,\overline{\varphi}_m(x)\big\}
\end{align*}
for all $x\in\overline{Q}$. This can be seen by modifying the definition of $\tilde{\rho}$ in \eqref{eq:definition_of_rho_tilde} accordingly. 

Note also that the just presented proof gives a uniform bound on the set of solutions to \eqref{eq:ako_poisson_problem}, $\underline{\varphi}\leq \varphi\leq\overline{\varphi}$ in $\cC^{0,\tau}(\overline{Q})$.

Combining these two observations one can carry out a proof of Theorem \ref{thm:envelope_theorem} similar to the original proof in \cite{ako_dirichlet_1961} but again without the ``complete type condition''.
\end{remark}

\begin{lemma}\label{lem:A1}
If $\varphi\in\hoeld{0,\tau}$, then $\tilde{\rho}(\cdot,\varphi)$ with $\tilde{\rho}$ defined in \eqref{eq:definition_of_rho_tilde} is contained in $\hoeld{0,\tau^2}$ and
\begin{equation}\label{eq:rho_inequaltiy}
\norm{\tilde{\rho}(\cdot,\varphi)}_\hoeld{0,\tau^2}\leq C\left(1+\norm{\varphi}_\hoeld{0,\tau}\right)^\tau
\end{equation}
for a constant $C>0$ independent of $\varphi$.
\end{lemma}
\begin{proof}
Recall that $\tilde{\rho}$ is bounded and $\tau$-H\"older continuous. Denote by $\tilde{c}:=[\tilde{\rho}]_{\cC^{0,\tau}(\overline{Q}\times\R)}$ the corresponding H\"older seminorm. For $x,y\in\overline{Q}$ there holds
\begin{align*}
\abs{\tilde{\rho}(x,\varphi(x))-\tilde{\rho}(y,\varphi(y))}&\leq \tilde{c}\abs{(x,\varphi(x))-(y,\varphi(y))}^\tau\\
&\leq \tilde{c}\left(\abs{x-y}+\norm{\varphi}_\hoeld{0,\tau}\abs{x-y}^\tau\right)^\tau\\
&\leq \tilde{c} \left(\diam(Q)^{(1-\tau)}+\norm{\varphi}_\hoeld{0,\tau}\right)^\tau \abs{x-y}^{\tau^2}.
\end{align*}
The statement follows.
\end{proof}
\begin{lemma}\label{lem:A2}
The map $\hoeld{0,\tau}\ni \varphi\mapsto \tilde{\rho}(\cdot,\varphi)\in \hoeld{0,\alpha}$ is continuous for every $\alpha\in (0,\tau^2)$.
\end{lemma}
\begin{proof}
This is shown in more generality in \cite[Proposition 6.2]{delaLlave}, except that the composition operator there is autonomous. Anyway, the proof for the case we need here is not too long. Let $\alpha\in(0,\tau^2)$ and $f\in\hoeld{0,\tau^2}$ be an arbitrary function. Then
\begin{align*}
[f]_\hoeld{0,\alpha}&=\sup_{x,y\in Q,x\neq y}\abs{f(x)-f(y)}^{1-\frac{\alpha}{\tau^2}}\frac{\abs{f(x)-f(y)}^{\frac{\alpha}{\tau^2}}}{\abs{x-y}^\alpha}\\
&\leq 2\norm{f}^{1-\frac{\alpha}{\tau^2}}_\hoeld{0}[f]_\hoeld{0,\tau^2}^\frac{\alpha}{\tau^2},
\end{align*}
which implies
\begin{equation}\label{eq:interpolation_inequality}
\norm{f}_\hoeld{0,\alpha}\leq 2\norm{f}^{1-\frac{\alpha}{\tau^2}}_\hoeld{0}\norm{f}_\hoeld{0,\tau^2}^\frac{\alpha}{\tau^2}.
\end{equation}
Now the stated continuity can be seen by an application of \eqref{eq:interpolation_inequality} to the function $f=\tilde{\rho}(\cdot,\varphi_1)-\tilde{\rho}(\cdot,\varphi_2)$, $\varphi_1,\varphi_2\in\hoeld{0,\tau}$, which as we know from Lemma \ref{lem:A1} is contained in $\hoeld{0,\tau^2}$, as well as the use of \eqref{eq:rho_inequaltiy} and
\begin{align*}
\norm{\tilde{\rho}(\cdot,\varphi_1)-\tilde{\rho}(\cdot,\varphi_2)}_\hoeld{0}\leq [\tilde{\rho}]_{\cC^{0,\tau}(\overline{Q}\times\R)}\norm{\varphi_1-\varphi_2}_\hoeld{0}^\tau.
\end{align*}
\end{proof}

\vspace{30pt}
\noindent\textbf{Acknowledgements.} 
The authors would like to thank V.I.Il'gisonis (State Atomic Energy Corporation Rosatom) and V.V.Ve\-de\-nya\-pin (Keldysh Institute of Applied Mathematics) for useful discussions.
B.Geb\-hard is very grateful for the hospitality of A.L.Skubachevskii, Y.O.Belyaeva and the RUDN group of Differential Equations during his two G-RISC research visits in Moscow.


%

\vspace{20pt}
\noindent Yulia O. Belyaeva\\
RUDN University, 6, Mikluhko--Maklaya str., Moscow, Russia\\
yilia-b@yandex.ru
\\

\noindent Bj\"orn Gebhard\\
Universit\"at Leipzig,  Augustusplatz 10, 04109 Leipzig,  Germany\\
bjoern.gebhard@math.uni-leipzig.de
 \\
 
\noindent Alexander L. Skubachevskii\\
RUDN University, 6, Mikluhko--Maklaya str., Moscow, Russia\\
skub@lector.ru 
%
\end{document}